\documentclass[12pt]{amsart}

\usepackage{bbm,bm}
\usepackage{rotating}
\usepackage{amssymb}
\usepackage[all]{xy}

\usepackage{color}

\usepackage[bookmarks=false,colorlinks,linkcolor=blue,citecolor=green]{hyperref}

\date{\today}

\theoremstyle{plain}
\newtheorem{theorem}{Theorem}
\newtheorem{corollary}[theorem]{Corollary}
\newtheorem{lemma}[theorem]{Lemma}
\newtheorem{proposition}[theorem]{Proposition}
\newtheorem{conjecture}[theorem]{Conjecture}

\theoremstyle{definition}
 
\newtheorem{remark}[theorem]{Remark} 

\theoremstyle{remark}
\newtheorem*{associativity}{Associativity}

\newtheorem*{compatibility}{Compatibility}

\newtheorem*{warning}{Warning}

\def\id{\mathrm{id}}
\def\Id{\mathrm{Id}}

\DeclareMathOperator{\End}{End}

\def\llb{{[\![}}
\def\rrb{{]\!]}}

\newcommand{\abs}[1]{\lvert#1\rvert} 
\newcommand{\pmat}[1]{\begin{pmatrix}#1\end{pmatrix}} 

\newcommand{\gel}{>_{\ell}}
\newcommand{\lel}{<_{\ell}}

\newcommand{\one}{\mathbf{1}}
\newcommand{\euler}{\mathbf{e}}

\let\onto=\twoheadrightarrow
\let\into=\hookrightarrow
\let\map=\xrightarrow
\newcommand{\xyinc}{\ar@{^{(}->}}

\newcommand{\type}[2]{\mathsf T_{#1}(#2)}

\def\field{\Bbbk}
\newcommand{\Sr}{\mathrm{S}} 

\newcommand{\rL}{\mathrm{L}}
\newcommand{\rG}{\mathrm{G}}
\newcommand{\rD}{\mathrm{D}}

\newcommand{\bL}{\mathbf L}

\newcommand{\bh}{\mathbf h}
\newcommand{\bk}{\mathbf k}

\newcommand{\bPi}{\mathbf \Pi}
\newcommand{\bG}{\mathbf{G}} 
\newcommand{\bp}{\mathbf p}
\newcommand{\bq}{\mathbf q}
\newcommand{\bmm}{\mathbf m}
\newcommand{\bd}{\mathbf d}

\newcommand{\Fc}{\mathcal{F}}
\newcommand{\Tc}{\mathcal{T}}
\newcommand{\Kc}{\mathcal{K}}
\newcommand{\Kcb}{\overline{\Kc}}

\newcommand{\Nb}{\mathbb{N}}

\newcommand{\Fb}{\mathbb{F}}

\newcommand{\Nf}{\mathfrak{n}}
\newcommand{\bNf}{\overline{\Nf}}
\newcommand{\Mf}{\mathfrak{m}}
\newcommand{\bMf}{\overline{\Mf}}

\newcommand{\GL}{\mathrm{GL}}
\newcommand{\Un}{\mathrm{U}}
\newcommand{\Mat}{\mathrm{M}}

\newcommand{\FC}{\bm{\mathrm{f}}}
\newcommand{\CF}{\bm{\mathrm{cf}}}
\newcommand{\SC}{\bm{\mathrm{scf}}}

\newcommand{\qand}{\quad\text{and}\quad}

\newcommand{\qqand}{\qquad\text{and}\qquad}

\author{Marcelo Aguiar}
\address[Aguiar]{
	Department of Mathematics\\
        Texas A\&M University\\
        College Station, TX 77843 
                }
\email{maguiar@math.tamu.edu}
\urladdr{http://www.math.tamu.edu/$\small\sim$maguiar}
\thanks{Aguiar supported in part by NSF grant DMS-1001935.}

\author{Nantel Bergeron}\address[Bergeron]
{Department of Mathematics and Statistics\\ York  University\\ To\-ron\-to, Ontario M3J 1P3\\ CANADA}
\email{bergeron@mathstat.yorku.ca}
\urladdr{http://www.math.yorku.ca/bergeron}
\thanks{Bergeron supported in part by CRC and NSERC}

 \author{Nathaniel Thiem}\address[Thiem]
 {Department of Mathematics\\University of Colorado\\ Boulder, CO 80309}
 \email{thiem@colorado.edu}
\urladdr{http://math.colorado.edu/$\small\sim$thiemn}
\thanks{Thiem supported in part by NSF FRG DMS-0854893}

\title[{H}opf monoids from functions on unitriangular matrices]{{H}opf monoids from  class functions on unitriangular matrices}

\keywords{Unitriangular matrix, class function, superclass function, Hopf monoid, Hopf algebra} 

\subjclass[2010]{05E05; 05E10; 05E15; 16T05; 16T30; 18D35; 20C33}

\begin{document}

\begin{abstract}
We build, from the collection of all groups
of unitriangular matrices, Hopf monoids in Joyal's
category of species. Such structure is carried by the
 collection of class function spaces on those groups,
and also by the collection of superclass function spaces,
 in the sense of Diaconis and Isaacs. 
Superclasses of unitriangular matrices admit a simple 
description from which we deduce a combinatorial model for the
Hopf monoid of superclass functions, in terms
of the Hadamard product of the Hopf monoids of linear orders and of
set partitions. This implies a recent result 
relating the Hopf algebra of superclass functions on unitriangular matrices
to symmetric functions in noncommuting variables.
We determine the algebraic structure of the Hopf monoid:
it is a free monoid in species, with the canonical Hopf structure.
As an application, we derive certain
estimates on the number of conjugacy classes of unitriangular matrices.
\end{abstract}

\maketitle

\section*{Introduction}

A Hopf monoid (in Joyal's category of species)
is an algebraic structure akin to that of a Hopf algebra.
Combinatorial structures which compose and decompose give rise
to Hopf monoids. These objects are the subject of~\cite[Part~II]{AM:2010}.
The few basic notions and examples needed for our purposes are reviewed in
Section~\ref{s:hopf}, including the Hopf monoids of linear orders,
set partitions, and simple graphs, and the Hadamard product of Hopf monoids.

The main goal of this paper is to construct a Hopf monoid out 
of the groups of unitriangular matrices with entries
in a finite field, and to do this in a transparent manner.
The structure exists on the collection of function spaces on these groups, and 
also on the collections of class function and superclass function spaces. It is induced by two
simple operations on this collection of groups: the passage from a matrix
to its principal minors gives rise to the product, and direct sum of matrices
gives rise to the coproduct.

Class functions are defined for arbitrary groups. 
An abstract notion and theory of superclass functions (and supercharacters)
for arbitrary groups exists~\cite{DI:2008}.
While a given group may admit several such theories,
there is a canonical choice of superclasses
for a special class of groups known as algebra groups.
These notions are briefly discussed in Section~\ref{ss:super-group}.
Unitriangular groups
are the prototype of such groups, and we employ the corresponding
notion of superclasses in Section~\ref{ss:super-unit}. The study of unitriangular
superclasses and supercharacters was initiated in~\cite{And:1995,And:1995-2},
making use of the Kirillov method~\cite{Kir:1995}, and by more elementary
means in~\cite{Yan:2001}.

Preliminaries on unitriangular matrices are discussed in Section~\ref{s:unitriang}.
The Hopf monoids $\FC(\Un)$ of functions and $\CF(\Un)$ of class functions are constructed in Section~\ref{s:hopf-class}.
The nature of the construction is fairly general; in particular, the same procedure
yields the Hopf monoid $\SC(\Un)$ of superclass functions in Section~\ref{ss:super-unit}. 

Unitriangular matrices over $\Fb_2$ may be identified with simple graphs,
and direct sums and the passage to principal minors correspond to simple
operations on graphs.
This yields a combinatorial model for $\FC(\Un)$ 
in terms of the Hadamard product of the Hopf monoids of linear orders and
of graphs, as discussed in Section~\ref{ss:f-model}.
The conjugacy classes on
the unitriangular groups exhibit great complexity and considerable attention has
been devoted to their study~\cite{Goo:2006,Hig:1960,Kir:1995,VAOV:2008}. 
We do not attempt an explicit combinatorial description of the Hopf monoid 
$\CF(\Un)$. On the other hand, superclasses are well-understood  (Section~\ref{ss:super-class}),  
and such a combinatorial description exists for $\SC(\Un)$. 
In Section~\ref{ss:super-structure}, we obtain a combinatorial model in terms of
the Hadamard product of the Hopf monoids of linear orders and of set partitions.
This has as a consequence the main result of~\cite{AIM:2012},
as we explain in Section~\ref{ss:hopfalg}.

Employing the combinatorial models, we derive
structure theorems for the Hopf monoids
$\FC(\Un)$ and $\SC(\Un)$ in Section~\ref{s:free}. 
Our main results state that both are free monoids 
with the canonical Hopf structure (in which the generators are primitive). 

Applications are presented in Section~\ref{s:additional}.
With the aid of Lagrange's theorem for Hopf monoids, one may derive
estimates on the number of conjugacy classes of unitriangular matrices,
in the form of certain recursive inequalities .
We obtain this application in Section~\ref{ss:counting}, where we also formulate
a refinement of Higman's conjecture on the polynomiality of these numbers. 
Other applications
involving the Hopf algebra of superclass functions of~\cite{AIM:2012}
are given in Section~\ref{ss:hopfalg}.

\smallskip

We employ two fields: the \emph{base} field $\field$ and the field of \emph{matrix entries} $\Fb$. We consider algebras and groups of matrices with entries in $\Fb$; all
other vector spaces are over $\field$. The field of matrix entries is often assumed
to be finite, and sometimes to be $\Fb_2$.

\subsection*{Acknowledgements} We thank the referee and Franco Saliola
for useful comments and suggestions.

\section{Hopf monoids}\label{s:hopf}

We review the basics on Hopf monoids and recall three examples built from linear orders, set partitions, and simple graphs respectively. We also consider the Hadamard product of Hopf monoids. 
In later sections, Hopf monoids are built from functions on unitriangular
matrices. The constructions of this section will allow us to provide
combinatorial models for them.

\subsection{Species and Hopf monoids}\label{ss:sp-hopf}
For the precise definition of \emph{vector species} and \emph{Hopf monoid}
we refer to~\cite[Chapter~8]{AM:2010}. The main ingredients are reviewed below.

A \emph{vector species} $\bp$ is a collection of vector spaces $\bp[I]$,
one for each finite set $I$, equivariant with respect to bijections $I\cong J$.
A morphism of species $f:\bp\to\bq$ is a collection of linear maps
$f_I:\bp[I]\to\bq[I]$ which commute with bijections.

A \emph{decomposition} of a finite set $I$ is a finite sequence $(S_1,\ldots,S_k)$
of disjoint subsets of $I$ whose union is $I$. In this situation, we write
\[
I=S_1\sqcup\cdots\sqcup S_k.
\]

A \emph{Hopf monoid} consists of a vector species $\bh$
equipped with two collections $\mu$ and $\Delta$ of linear maps
\[
\bh[S_1]\otimes\bh[S_2] \map{\mu_{S_1,S_2}}\bh[I]
\qand
\bh[I]\map{\Delta_{S_1,S_2}} \bh[S_1]\otimes\bh[S_2].
\]
There is one map in each collection for each finite set $I$ and each decomposition $I=S_1\sqcup S_2$. This data is subject to a number of axioms, of which the main ones follow.

\begin{associativity}
For each decomposition $I=S_1\sqcup S_2\sqcup S_3$, diagrams
\begin{gather}\label{e:assoc}
\begin{gathered}
\xymatrix@R+1pc@C+20pt{
\bh[S_1]\otimes\bh[S_2]\otimes\bh[S_3]\ar[r]^-{\id\otimes\mu_{S_2,S_3}}
\ar[d]_{\mu_{S_1,S_2}\otimes\id} & \bh[S_1]\otimes\bh[S_2\sqcup S_3]\ar[d]^{\mu_{S_1,S_2\sqcup S_3}} \\
\bh[S_1\sqcup S_2]\otimes \bh[S_3]\ar[r]_-{\mu_{S_1\sqcup S_2,S_3}} & \bh[I]
}
\end{gathered}
\\
\label{e:coassoc}
\begin{gathered}
\xymatrix@R+1pc@C+40pt{
\bh[I]\ar[r]^-{\Delta_{S_1\sqcup S_2,S_3}} \ar[d]_{\Delta_{S_1,S_2\sqcup S_3}}
& \bh[S_1\sqcup S_2]\otimes \bh[S_3]\ar[d]^{\Delta_{S_1,S_2}\otimes\id}
\\
\bh[S_1]\otimes\bh[S_2\sqcup S_3]\ar[r]_-{\id\otimes\Delta_{S_2,S_3}}
& \bh[S_1]\otimes\bh[S_2]\otimes\bh[S_3]
}
\end{gathered}
\end{gather}
commute.
\end{associativity}

\begin{compatibility}
Fix decompositions $S_1\sqcup S_2=I=T_1\sqcup T_2$,
and consider the resulting pairwise intersections:
\[
A:=S_1\cap T_1,\ B:=S_1\cap T_2,\ C:=S_2\cap T_1,\ D:=S_2\cap T_2,
\]
as illustrated below.
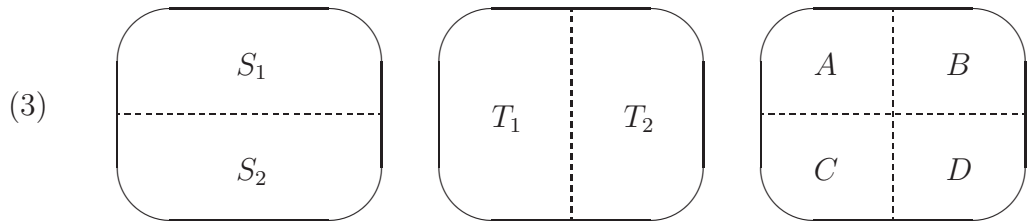
\begin{equation}\label{e:4sets}
\begin{gathered}
\begin{picture}(100,90)(20,0)
  \put(50,40){\oval(100,80)}
  \put(0,40){\dashbox{2}(100,0){}}
  \put(45,55){$S_1$}
  \put(45,15){$S_2$}
\end{picture}
\quad
\begin{picture}(100,90)(10,0)
  \put(50,40){\oval(100,80)}
  \put(50,0){\dashbox{2}(0,80){}}
  \put(20,35){$T_1$}
  \put(70,35){$T_2$}
\end{picture}
\quad
\begin{picture}(100,90)(0,0)
\put(50,40){\oval(100,80)}
  \put(0,40){\dashbox{2}(100,0){}}
  \put(50,0){\dashbox{2}(0,80){}}
  \put(20,55){$A$}
  \put(70,55){$B$}
  \put(20,15){$C$}
  \put(70,15){$D$}
\end{picture}
\end{gathered}
\end{equation}

For any such pair of decompositions, the diagram
\begin{equation}\label{e:comp}
\begin{gathered}
\xymatrix@R+2pc@C-5pt{
\bh[A] \otimes \bh[B] \otimes \bh[C] \otimes \bh[D] \ar[rr]^{\cong} & &
\bh[A] \otimes \bh[C] \otimes \bh[B] \otimes \bh[D] \ar[d]^{\mu_{A,C}
\otimes \mu_{B,D}}\\
\bh[S_1] \otimes \bh[S_2] \ar[r]_-{\mu_{S_1,S_2}}\ar[u]^{\Delta_{A,B} \otimes
\Delta_{C,D}} & \bh[I] \ar[r]_-{\Delta_{T_1,T_2}} & \bh[T_1] \otimes
\bh[T_2]
}
\end{gathered}
\end{equation}
must commute. The top arrow stands for the map that interchanges the middle factors.
\end{compatibility}

In addition, the Hopf monoid $\bh$ is \emph{connected} if
 $\bh[\emptyset]=\field$ and the maps
\[
\xymatrix{
\bh[I]\otimes\bh[\emptyset]  \ar@<0.5ex>[r]^-{\mu_{I,\emptyset}} & 
\bh[I] \ar@<0.5ex>[l]^-{\Delta_{I,\emptyset}} 
}
\qqand
\xymatrix{
\bh[\emptyset]\otimes\bh[I]  \ar@<0.5ex>[r]^-{\mu_{\emptyset,I}} & 
\bh[I] \ar@<0.5ex>[l]^-{\Delta_{\emptyset,I}} 
}
\]
are the canonical identifications.

The collection $\mu$ is the \emph{product} and the collection $\Delta$
is the \emph{coproduct} of the Hopf monoid $\bh$.

A Hopf monoid is (co)commutative if the left (right) diagram below commutes
for all decompositions $I=S_1\sqcup S_2$.
\begin{equation}\label{e:comm}
\begin{gathered}
\xymatrix@C-15pt{
\bh[S_1]\otimes\bh[S_2] \ar[rr]^-{\cong} \ar[rd]_{\mu_{S_1,S_2}}  & & \bh[S_2]\otimes\bh[S_1] \ar[ld]^{\mu_{S_2,S_1}} \\
& \bh[I] &
}
\qquad
\xymatrix@C-15pt{
\bh[S_1]\otimes\bh[S_2] \ar[rr]^-{\cong}   & &\bh[S_2]\otimes\bh[S_1]  \\
& \bh[I] \ar[lu]^{\Delta_{S_1,S_2}} \ar[ru]_{\Delta_{S_2,S_1}} &
}
\end{gathered}
\end{equation}
The top arrows stand for
the map that interchanges the factors.

A morphism of Hopf monoids $f:\bh\to\bk$ is a morphism of species that
commutes with $\mu$ and $\Delta$.

\subsection{The Hopf monoid of linear orders}\label{ss:order}

For any finite set $I$,
$\rL[I]$ is the set of all linear orders on $I$. For instance, if $I=\{a,b,c\}$,
\[
\rL[I]=\{abc,\,bac,\,acb,\,bca,\,cab,\,cba\}.
\]
Let $\bL[I]$ denote the vector space with basis $\rL[I]$. The collection $\bL$ is a vector species.

Let $I=S_1\sqcup S_2$.
Given linear orders $\ell_i$ on $S_i$, $i=1,2$,
their concatenation $\ell_1\cdot \ell_2$ 
is a linear order on $I$. This is the list consisting of the elements
of $S_1$ as ordered by $\ell_1$, followed by those of $S_2$ as ordered by $\ell_2$.
Given a linear order $\ell$ on $I$ and $S\subseteq I$, the restriction 
 $\ell|_S$ (the list consisting of the elements of $S$ written in the order
 in which they appear in $\ell$) is a linear order on $S$.
These operations give rise to maps
\begin{equation}\label{e:L}
\begin{array}{rcl}
\rL[S_1]\times\rL[S_2] & \to & \rL[I] \\
(\ell_1,\ell_2) & \mapsto & \ell_1\cdot\ell_2
\end{array}
\qquad
\begin{array}{rcl}
 \rL[I]  & \to & \rL[S_1]\times\rL[S_2]\\
  \ell  & \mapsto& (\ell |_{S_1},\ell |_{S_2}).
\end{array}
\end{equation}
Extending by linearity, we obtain linear maps
\[
 \mu_{S_1,S_2}:\bL[S_1]\otimes\bL[S_2] \to \bL[I]
 \qand
\Delta_{S_1,S_2}:\bL[I]\to \bL[S_1]\otimes\bL[S_2]
\]
which turn $\bL$ into a Hopf monoid. For instance,
given linear orders
$\ell_i$ on $S_i$, $i=1,2$, the commutativity of~\eqref{e:comp} boils down to the fact that the concatenation of $\ell_1|_A$ and $\ell_2|_C$ agrees with the restriction to $T_1$ of $\ell_1\cdot \ell_2$. The Hopf monoid $\bL$ is cocommutative but not commutative.
For more details, see~\cite[Section~8.5]{AM:2010}.

\subsection{The Hopf monoid of set partitions}\label{ss:partition}

A \emph{partition} of a finite set $I$ is a collection $X$ of disjoint nonempty subsets whose union is $I$. The subsets are the \emph{blocks} of $X$.

Given a partition $X$ of $I$ and $S\subseteq I$, the restriction
$X|_S$ is the partition of $S$ whose blocks are 
the nonempty intersections of the blocks of $X$ with $S$.
Let $I=S_1\sqcup S_2$.
Given partitions $X_i$ of $S_i$, $i=1,2$, 
their union is the partition $X_1\sqcup X_2$ of $I$ 
whose blocks are the blocks of $X_1$ and the blocks of $X_2$.
A quasi-shuffle of $X_1$ and $X_2$ is any partition of $I$ whose
restriction to $S_i$ is $X_i$, $i=1,2$. 

Let $\Pi[I]$ denote the set of partitions of $I$ and $\bPi[I]$ the vector
space with basis $\Pi[I]$. 
A Hopf monoid structure on $\bPi$ is defined and studied in~\cite[Section~12.6]{AM:2010}. Among its various linear bases, we are interested in the basis $\{m_X\}$
on which the operations are as follows. 
The product
\[
\mu_{S_1,S_2}: \bPi[S_1] \otimes \bPi[S_2] \to \bPi[I] 
\]
is given by
\begin{equation}\label{e:prod-m}
\mu_{S_1,S_2}(m_{X_1} \otimes m_{X_2}) = \sum_{X:  \  X|_{S_1}=X_1 \text{ and } X|_{S_2}=X_2 } m_X.
\end{equation}
The coproduct
\[
\Delta_{S_1,S_2}: \bPi[I]  \to \bPi[S_1]\otimes\bPi[S_2]
\]
is given by
\begin{equation}\label{e:coprod-m}
\Delta_{S_1,S_2}(m_X) = 
\begin{cases}
m_{X|_{S_1}} \otimes m_{X|_{S_2}}  & \text{if $S_1$ is the union of some blocks of $X$,} \\
0 & \text{otherwise.}
\end{cases}
\end{equation}
Note that the following conditions are equivalent, for a partition $X$ of $I$.
\begin{itemize}
\item $S_1$ is the union of some blocks of $X$.
\item $S_2$ is the union of some blocks of $X$.
\item $X=X|_{S_1}\sqcup X|_{S_2}$.
\end{itemize}

These operations turn the species $\bPi$ into a Hopf monoid that is both commutative and cocommutative.

\subsection{The Hopf monoid of simple graphs}\label{ss:graph}
A \emph{(simple) graph} $g$ on a finite set $I$ is a relation on $I$ that is symmetric and
irreflexive. The elements of $I$ are the \emph{vertices} of $g$
There is an \emph{edge} between two vertices when they are related by $g$.

 Given a graph $g$ on $I$ and $S\subseteq I$, the restriction
$g|_S$ is the graph on $S$ whose edges are the edges of $g$ between elements of $S$.
Let $I=S_1\sqcup S_2$.
Given graphs $g_i$ of $S_i$, $i=1,2$, 
their union is the graph $g_1\sqcup g_2$ of $I$ 
whose edges are those of $g_1$ and those of $g_2$.
A quasi-shuffle of $g_1$ and $g_2$ is any graph on $I$ whose
restriction to $S_i$ is $g_i$, $i=1,2$. 

Let $\rG[I]$ denote the set of graphs on $I$ and $\bG[I]$ the vector
space with basis $\rG[I]$. 
A Hopf monoid structure on $\bG$ is defined and studied in~\cite[Section~13.2]{AM:2010}. We are interested in the basis $\{m_g\}$
on which the operations are as follows. 
The product
\[
\mu_{S_1,S_2}: \bG[S_1] \otimes \bG[S_2] \to \bG[I] 
\]
is given by
\begin{equation}\label{e:prod-g}
\mu_{S_1,S_2}(m_{g_1} \otimes m_{g_2}) = \sum_{g: \  g|_{S_1}=g_1 \text{ and }  g|_{S_2}=g_2} m_g.
\end{equation}
The coproduct
\[
\Delta_{S_1,S_2}: \bG[I]  \to \bG[S_1]\otimes\bG[S_2]
\]
is given by
\begin{equation}\label{e:coprod-g}
\Delta_{S_1,S_2}(m_g) = 
\begin{cases}
m_{g|_{S_1}} \otimes m_{g|_{S_2}}  & \text{if no edge of $g$ connects $S_1$ to $S_2$,} \\
0 & \text{otherwise.}
\end{cases}
\end{equation}
Note that no edge of $g$ connects $S_1$ to $S_2$ if and only if 
$g=g|_{S_1}\sqcup g|_{S_2}$.

These operations turn the species $\bG$ into a Hopf monoid that is both commutative and cocommutative.

\begin{remark}\label{r:selfdual}
The dual of a species $\bp$ is the collection $\bp^*$ of dual vector spaces:
$\bp^*[I]=\bp[I]^*$.
A species $\bp$ is said to be finite-dimensional if each space $\bp[I]$
is finite-dimensional.
Dualizing the operations of a finite-dimensional Hopf monoid $\bh$, 
one obtains a Hopf monoid $\bh^*$. The Hopf monoid $\bh$ is called self-dual
if $\bh\cong\bh^*$. In general, such isomorphism is not unique.

Over a field of characteristic $0$, a Hopf monoid that is connected, commutative and cocommutative, is always
self-dual. This is a consequence of the Cartier-Milnor-Moore theorem.
(The isomorphism with the dual is not canonical.)

In particular, the Hopf monoids $\bPi$ and $\bG$ are self-dual.
In~\cite{AM:2010}, the preceding descriptions of these Hopf monoids are
stated in terms of their duals $\bPi^*$ and $\bG^*$. 
A different description of $\bPi$ is given in~\cite[Section~12.6.2]{AM:2010}.
To reconcile the two, one should use the explicit isomorphism
$\bPi\cong\bPi^*$ given in~\cite[Proposition~12.48]{AM:2010}.
\end{remark}

\subsection{The Hadamard product}\label{ss:hadamard}

Given species $\bp$ and $\bq$, their \emph{Hadamard product} is the species $\bp\times\bq$ defined by
\[
(\bp\times\bq)[I] = \bp[I]\otimes\bq[I].
\]
If $\bh$ and $\bk$ are Hopf monoids, then so is $\bh\times\bk$, with the following operations. Let $I=S_1\sqcup S_2$. The product is
\[
\xymatrix@C+12pt{
(\bh\times\bk)[S_1]\otimes(\bh\times\bk)[S_2] \ar@{=}[d] \ar@{-->}[rr] &
& (\bh\times\bk)[I] \ar@{=}[d] \\
\bh[S_1]\otimes \bk[S_1]\otimes \bh[S_2]\otimes \bk[S_2]
\ar[r]_{\cong} & 
\bh[S_1]\otimes \bh[S_2]\otimes \bk[S_1]\otimes \bk[S_2]
\ar[r]_-{\mu_{S,T}\otimes\mu_{S,T}} &
\bh[I]\otimes\bk[I]
}
\]
and the coproduct is defined dually. If $\bh$ and $\bk$ are (co)commutative, then so is $\bh\times\bk$. For more details, see~\cite[Section~8.13]{AM:2010}.

\section{Unitriangular matrices}\label{s:unitriang}

This section sets up the basic notation pertaining to unitriangular
matrices, and discusses two simple but important constructions: direct sum of
matrices and the passage from a matrix to its principal minors.
The Hopf monoid constructions of later sections are based on them.
The key results are Lemmas~\ref{l:principal-unit} and~\ref{l:direct-principal}.
The former is the reason why we must use unitriangular matrices:
for arbitrary matrices, the passage to principal minors is not multiplicative.
The latter will be responsible (in later sections)
for the necessary compatibility between the product and coproduct of the Hopf
monoids.

\smallskip

Let $\Fb$ be a field, $I$ a finite set, and $\ell$ a linear order on $I$.
Let $\Mat(I)$ denote the algebra of matrices 
\[
\text{$A=(a_{ij})_{i,j\in I}$,
 $a_{ij}\in\Fb$ for all $i,j\in I$.}
\]
The general linear group $\GL(I)$ consists of the invertible matrices in $\Mat(I)$,
and the subgroup $\Un(I,\ell)$ consists of the \emph{upper $\ell$-unitriangular} matrices
\[
\text{$U=(u_{ij})_{i,j\in I}$,
$u_{ii}=1$ for all $i\in I$,
$u_{ij}=0$ whenever $i\gel j$.}
\]
If $\ell'$ is another linear order on $I$, then $\Un(I,\ell)$
and $\Un(I,\ell')$ are conjugate subgroups of $\GL(I)$.
However, we want to keep track of all groups in this collection, and
of the manner in which they interact.

\subsection{Direct sum of matrices}\label{ss:direct}

Let $I=S_1\sqcup S_2$ be a decomposition. Given  $A=(a_{ij})\in\Mat(S_1)$ and $B=(b_{ij})\in\Mat(S_2)$, their direct sum is the matrix
$A\oplus B=(c_{ij})\in\Mat(I)$ with entries
\[
c_{ij}= \begin{cases}
a_{ij} & \text{ if both $i,j\in S_1$,} \\
 b_{ij}        & \text{ if both $i,j\in S_2$,}\\
  0       & \text{ otherwise.}
\end{cases}
\]

Let $\ell\in\rL[I]$. The direct sum of an $\ell|_{S_1}$-unitriangular  and an
$\ell|_{S_2}$-unitriangular matrix is $\ell$-unitriangular.
The morphism of algebras
\[
\Mat(S_1)\times\Mat(S_2) \to \Mat(I),
\quad (A,B) \mapsto A\oplus B
\]
thus restricts to a morphism of groups
\begin{equation}\label{e:sigma}
\sigma_{S_1,S_2} :\Un(S_1,\ell|_{S_1})\times \Un(S_2,\ell|_{S_2}) \to \Un(I,\ell).
\end{equation}
(The dependence of $\sigma_{S_1,S_2}$ on $\ell$ is left implicit).

Direct sum of matrices is associative; thus, for any decomposition $I=S_1\sqcup S_2\sqcup S_3$ 
the diagram
\begin{equation}\label{e:sigma-asso}
\begin{gathered}
\xymatrix@C+15pt{
\Un(S_1,\ell|_{S_1})\times\Un(S_2,\ell|_{S_2})\times\Un(S_3,\ell|_{S_3}) \ar[d]_{\id\times\sigma_{S_2,S_3}}
\ar[r]^-{\sigma_{S_1,S_2}\times\id} & \Un(S_1\sqcup S_2,\ell|_{S_1\sqcup S_2})\times\Un(S_3,\ell|_{S_3})
\ar[d]^{\sigma_{S_1\sqcup S_2,S_3}} \\
\Un(S_1,\ell|_{S_1})\times\Un(S_2\sqcup S_3,\ell|_{S_2\sqcup S_3}) \ar[r]_-{\sigma_{S_1,S_2\sqcup S_3}}
& \Un(I,\ell)
}
\end{gathered}
\end{equation}
commutes. Note also that, with these definitions, $A\oplus B$ and $B\oplus A$
are the same matrix. Thus, the following diagram commutes.
\begin{equation}\label{e:sigma-comm}
\begin{gathered}
\xymatrix@C-15pt{
\Un(S_1,\ell|_{S_1})\times\Un(S_2,\ell|_{S_2}) \ar[rr]^-{\cong} \ar[rd]_{\sigma_{S_1,S_2}} & & \Un(S_2,\ell|_{S_2})\times\Un(S_1,\ell|_{S_1}) \ar[ld]^{\sigma_{S_2,S_1}} \\
& \Un(S_1\sqcup S_2,\ell)&
}
\end{gathered}
\end{equation}


\subsection{Principal minors}\label{ss:principal}

Given  $A=(a_{ij})\in\Mat(I)$, the \emph{principal minor} indexed by $S\subseteq I$
is the matrix 
\[
A_S =(a_{ij})_{i,j\in S}.
\]
In general $A_S$ is not invertible even if $A$ is.
In addition, the assignment $A\mapsto A_S$
does not preserve multiplications.  
On the other hand, if $U$ is $\ell$-unitriangular, then
$U_S$ is $\ell|_S$-unitriangular. In regard to multiplicativity, we have the following
basic fact. 

We say that $S$ is an $\ell$-segment if $i,k\in S$ and $i\lel j\lel k$ 
imply that also $j\in S$.

Let $E_{ij}\in\Mat(I)$ denote the \emph{elementary} matrix in which the $(i,j)$ entry 
is $1$ and all other entries are $0$.

\begin{lemma}\label{l:principal-unit}
Let $\ell\in\rL[I]$ and $S\subseteq I$. The map
\[
\Un(I,\ell) \to \Un(S,\ell|_S),\quad U\mapsto U_S
\]
is a morphism of groups if and only if $S$ is an $\ell$-segment.
\end{lemma}
\begin{proof}
 Suppose the map
is a morphism of groups. Let $i,j,k\in I$ be such that $i,k\in S$ and $i\lel j\lel k$.
The matrices
\[
\Id+E_{ij} \qand \Id+E_{jk}
\]
are in $\Un(I,\ell)$ and
\[
(\Id+E_{ij})\cdot(\Id+E_{jk}) = \Id + E_{ij} + E_{jk} + E_{ik}.
\]
If $j\notin S$, then the two matrices are in the kernel of the map, while their product
is mapped to $\Id+E_{ik}\neq \Id$. Thus, $j\in S$ and $S$ is an $\ell$-segment.

The converse implication follows from the fact that if $U$ and $V$ are $\ell$-unitriangular, then the $(i,k)$-entry of $UV$ is
\[
\sum_{i\leq_{\ell} j \leq_{\ell} k} u_{ij}v_{jk}. \qedhere
\]
\end{proof}

Let $I=S_1\sqcup S_2$ be a decomposition, $\ell_i\in\rL[S_i]$, $i=1,2$.
We define a map
\begin{equation}\label{e:pi}
\pi_{S_1,S_2}: \Un(I,\ell_1\cdot\ell_2) \to  \Un(S_1,\ell_1)\times  \Un(S_2,\ell_2)
\end{equation}
by
\[
U \mapsto (U_{S_1},U_{S_2}).
\]
Note that $S_1$ is an initial segment for $\ell_1\cdot\ell_2$ and $S_2$ is a final 
segment for $\ell_1\cdot\ell_2$. Thus $\pi_{S_1,S_2}$ is a morphism of groups
by Lemma~\ref{l:principal-unit}.

If $R\subseteq S\subseteq I$, then $(A_{S})_{R} = A_{R}$.
This implies the following commutativity, for any decomposition
 $I=S_1\sqcup S_2\sqcup S_3$ and $\ell_i\in\rL[S_i]$, $i=1,2,3$.
\begin{equation}\label{e:pi-asso}
\begin{gathered}
\xymatrix@C+15pt{
\Un(I,\ell_1\cdot\ell_2\cdot\ell_3) \ar[r]^-{\pi_{S_1\sqcup S_2,S_3}}
\ar[d]_-{\pi_{S_1,S_2\sqcup S_3}} & 
\Un(S_1\sqcup S_2,\ell_1\cdot\ell_2)\times\Un(S_3,\ell_3) \ar[d]^-{\pi_{S_1,S_2}\times\id} \\
\Un(S_1,\ell_1)\times\Un(S_2\sqcup S_3,\ell_2\cdot\ell_3) 
\ar[r]_-{\id\times\pi_{S_2,S_3}} &
\Un(S_1,\ell_1)\times\Un(S_2,\ell_2)\times\Un(S_3,\ell_3) 
}
\end{gathered}
\end{equation}

\subsection{Direct sums and principal minors}\label{ss:direct-principal}

The following key result relates the collection of morphisms $\sigma$ to
the collection $\pi$.

\begin{lemma}\label{l:direct-principal}
Fix two decompositions $I=S_1\sqcup S_2=T_1\sqcup T_2$
and let $A$, $B$, $C$ and $D$ be the resulting intersections, as in~\eqref{e:4sets}.
Let $\ell_i$ be a linear order on $S_i$, $i=1,2$, and $\ell=\ell_1\cdot\ell_2$.
Then the following diagram commutes.
\begin{equation}\label{e:direct-principal}
\begin{gathered}
\xymatrix@C+17pt{
\Un(T_1,\ell|_{T_1}) \times \Un(T_2,\ell|_{T_2}) \ar[d]_-{\sigma_{T_1,T_2}} \ar[r]^-{\pi_{A,C}\times\pi_{B,D}} & \Un(A,\ell_1|_{A})\times\Un(C,\ell_2|_{C})\times\Un(B,\ell_1|_{B})\times\Un(D,\ell_2|_{D}) \ar[dd]^-{\cong}\\
\Un(I,\ell) \ar[d]_-{\pi_{S_1,S_2}} & \\
\Un(S_1,\ell_1)\times\Un(S_2,\ell_2) 
 & \Un(A,\ell_1|_{A})\times\Un(B,\ell_1|_{B})\times\Un(C,\ell_2|_{C})\times\Un(D,\ell_2|_{D}) \ar[l]^-{\sigma_{A,B}\times\sigma_{C,D}}
}
\end{gathered}
\end{equation}
\end{lemma}
\begin{proof}
First note that since $\ell|_{T_1} = (\ell_1|_{A})\cdot(\ell_2|_{C})$, $\pi_{A,C}$ does
map as stated in the diagram, and similarly for $\pi_{B,D}$.
The commutativity of the diagram boils down to the simple fact that
\[
(U\oplus V)_{S_1} = U_A\oplus V_B
\]
(and a similar statement for $S_2$, $C$, and $D$).
This holds for any $U\in\Mat(T_1)$ and $V\in\Mat(T_2)$.
\end{proof}

\section{A Hopf monoid of (class) functions}\label{s:hopf-class}

We employ the operations of Section~\ref{s:unitriang}
(direct sum
of matrices and the passage from a matrix to its principal minors) 
to build a Hopf monoid structure on the
collection of function spaces on unitriangular matrices.
The collection of class function spaces gives rise to a Hopf submonoid.
With matrix entries in $\Fb_2$, the Hopf monoid of functions may be identified with the Hadamard product of the Hopf monoids of linear orders and of simple graphs.

\subsection{Functions}\label{ss:function-group}

Given a set $X$, let $\FC(X)$ denote the vector space of
functions on $X$ with values on the base field $\field$.
The functor
\[
\FC: \{\text{sets}\} \to \{\text{vector spaces}\}
\]
is contravariant. 
If at least one of two sets $X_1$ and $X_2$
is finite, then there is a canonical isomorphism
\begin{equation}\label{e:f-strong}
\FC(X_1\times X_2) \cong \FC(X_1)\otimes\FC(X_2).
\end{equation}
A function $f\in \FC(X_1\times X_2)$ corresponds to $\sum_i f^1_i\otimes f^2_i\in \FC(X_1)\otimes\FC(X_2)$ if and only if
\[
f(x_1,x_2) = \sum_i f^1_i(x_1)f^2_i(x_2)
\quad \forall\,x_1\in X_1,\, x_2\in X_2.
\]

Given an element $x\in X$, let $\kappa_x:X\to\field$ denote its characteristic function:
\begin{equation}\label{e:char-function}
\kappa_x(y) =
\begin{cases}
1 & \text{ if }y=x, \\
0         & \text{ if not.}
\end{cases}
\end{equation}
Suppose now that $X$ is finite. As $x$ runs over the elements of $X$, the functions $\kappa_x$
form a linear basis of $\FC(X)$.
If $\varphi:X\to X'$ is a function and $x'$ is an element of $X'$, then
\begin{equation}\label{e:function-functor}
\kappa_{x'}\circ \varphi = \sum_{\varphi(x)=x'} \kappa_x.
\end{equation}
Under~\eqref{e:f-strong},
\begin{equation}\label{e:char-strong}
\kappa_{(x_1,x_2)} \leftrightarrow \kappa_{x_1}\otimes\kappa_{x_2}.
\end{equation}

\subsection{Class functions on groups}\label{ss:class-group}

Given a group $G$, let $\CF(G)$ denote the vector space of
\emph{class functions} on $G$. These are the functions $f:G\to\field$ that
are constant on conjugacy classes of $G$. 
If $\varphi:G\to G'$ is a morphism of groups and $f$ is a class function on $G'$,
then $f\circ \varphi$ is a class function on $G$. In this manner,
\[
\CF: \{\text{groups}\} \to \{\text{vector spaces}\}
\]
is a contravariant functor. If at least one of two groups $G_1$ and $G_2$
is finite, then there is a canonical isomorphism
\begin{equation}\label{e:strong}
\CF(G_1\times G_2) \cong \CF(G_1)\otimes\CF(G_2)
\end{equation}
obtained by restriction from the isomorphism~\eqref{e:f-strong}.

Given a conjugacy class $C$ of $G$, let $\kappa_C:G\to\field$ denote its characteristic function:
\begin{equation}\label{e:char-class}
\kappa_C(x) =
\begin{cases}
1 & \text{ if }x\in C, \\
0         & \text{ if not.}
\end{cases}
\end{equation}
Suppose $G$ has finitely many conjugacy classes.
As $C$ runs over the conjugacy classes of $G$, the functions $\kappa_C$
form a linear basis of $\CF(G)$. If $C'$ is a conjugacy class of $G'$, then
\begin{equation}\label{e:class-functor}
\kappa_{C'}\circ \varphi = \sum_{\varphi(C)\subseteq C'} \kappa_C.
\end{equation}
The conjugacy classes of $G_1\times G_2$ are of the form $C_1\times C_2$
where $C_i$ is a conjugacy class of $G_i$, $i=1,2$. Under~\eqref{e:strong},
\begin{equation}\label{e:class-strong}
\kappa_{C_1\times C_2} \leftrightarrow \kappa_{C_1}\otimes\kappa_{C_2}.
\end{equation}

\subsection{Functions on unitriangular matrices}\label{ss:f-unit}

We assume for the rest of this section that the field $\Fb$ of matrix entries is finite. Thus, all groups
$\Un(I,\ell)$ of unitriangular matrices are finite.

We define a vector species $\FC(\Un)$ as follows. On a finite set $I$,
\[
\FC(\Un)[I] = \bigoplus_{\ell\in\rL[I]} \FC\bigl(\Un(I,\ell)\bigr).
\]
In other words, $\FC(\Un)[I]$ is the direct sum of the spaces of functions on
all unitriangular groups on $I$. A bijection $\sigma:I\cong J$ induces an isomorphism
$\Un(I,\ell)\cong\Un(J,\sigma\cdot\ell)$ and therefore an isomorphism $\FC(\Un)[I]\cong\FC(\Un)[J]$. Thus, $\FC(\Un)$ is a species.

Let $I=S_1\sqcup S_2$ and $\ell_i\in\rL[S_i]$, $i=1,2$.
Applying the functor $\FC$ to the morphism $\pi_{S_1,S_2}$ in~\eqref{e:pi}
and composing with the isomorphism in~\eqref{e:f-strong},
we obtain a linear map
\[
\FC\bigl(\Un(S_1,\ell_1)\bigr)\otimes \FC\bigl(\Un(S_2,\ell_2)\bigr)
\to \FC\bigl(\Un(I,\ell_1\cdot\ell_2)\bigr).
\]
Adding over all $\ell_1\in\rL[S_1]$ and $\ell_2\in\rL[S_2]$, we obtain a linear map
\begin{equation}\label{e:f-prod}
\mu_{S_1,S_2} : \FC(\Un)[S_1]\otimes \FC(\Un)[S_2] \to \FC(\Un)[I].
\end{equation}
Explicitly, given functions $f:\Un(S_1,\ell_1)\to\field$ and $g:\Un(S_2,\ell_2)\to\field$, 
\[
\mu_{S_1,S_2}(f\otimes g): \Un(I,\ell_1\cdot\ell_2)\to\field
\]
is the function given by
\begin{equation}\label{e:f-prod2}
U \mapsto f(U_{S_1})g(U_{S_2}).
\end{equation}

Similarly, from the map $\sigma_{S_1,S_2}$ in~\eqref{e:sigma} we obtain 
the components
\[
 \FC\bigl(\Un(I,\ell)\bigr) \to  \FC\bigl(\Un(S_1,\ell|_{S_1})\bigr) \otimes  \FC\bigl(\Un(S_2,\ell|_{S_2})\bigr)
\]
(one for each $\ell\in\rL[I]$) of a linear map
\begin{equation}\label{e:f-coprod}
\Delta_{S_1,S_2}:   \FC(\Un)[I] \to \FC(\Un)[S_1]\otimes \FC(\Un)[S_2].
\end{equation}
Explicitly, given a function $f:\Un(I,\ell)\to\field$, we have $\Delta_{S_1,S_2}(f)=\sum_i f^1_i\otimes f^2_i$ where 
\[
f^1_i:\Un(S_1,\ell|_{S_1})\to\field
\qand
f^2_i:\Un(S_2,\ell|_{S_2})\to\field
\]
are  functions such that
\begin{equation}\label{e:f-coprod2}
f(U_1\oplus U_2) = \sum_i f^1_i(U_1) f^2_i(U_2) \quad \text{for all $U_1\in\Un(S_1,\ell|_{S_1})$
and $U_2\in\Un(S_2,\ell|_{S_2})$.}
\end{equation}

\begin{proposition}\label{p:f-unit}
With the operations~\eqref{e:f-prod} and~\eqref{e:f-coprod},
the species $\FC(\Un)$ is a connected Hopf monoid. It is cocommutative.
\end{proposition}
\begin{proof}
Axioms~\eqref{e:assoc},~\eqref{e:coassoc} and~\eqref{e:comp}
follow from~\eqref{e:sigma-asso},~\eqref{e:pi-asso} and~\eqref{e:direct-principal}
by functoriality. In the same manner, cocommutativity~\eqref{e:comm}
follows from~\eqref{e:sigma-comm}.
\end{proof}

We describe the operations on the basis of characteristic functions~\eqref{e:char-function}. Let $U_i\in \Un(S_i,\ell_i)$, $i=1,2$.
It follows from~\eqref{e:function-functor} and~\eqref{e:char-strong},
or from~\eqref{e:f-prod2}, that the product is
\begin{equation}\label{e:char-f-prod}
\mu_{S_1,S_2}(\kappa_{U_1}\otimes \kappa_{U_2}) = \sum_{\pi_{S_1,S_2}(U)= (U_1, U_2)} \kappa_U
= \sum_{U_{S_1}=U_1,\,U_{S_2}=U_2} \kappa_U.
\end{equation}
Similarly, the coproduct is
\begin{equation}\label{e:char-f-coprod}
\Delta_{S_1,S_2}(\kappa_{U}) = \sum_{\sigma_{S_1,S_2}(U_1, U_2)=U} \kappa_{U_1}\otimes \kappa_{U_2}
= \begin{cases}
\kappa_{U_{S_1}}\otimes \kappa_{U_{S_2}} & \text{ if } U= U_{S_1} \oplus U_{S_2}, \\
   0      & \text{ otherwise.}
\end{cases}
\end{equation}

\subsection{Constant functions}\label{ss:constant}
Let $\one_\ell$ denote the constant function on $\Un(I,\ell)$ with value $1$.
Let $I=S_1\sqcup S_2$. It follows from~\eqref{e:f-prod2}
that
\[
\mu_{S_1,S_2}(\one_{\ell_1}\otimes\one_{\ell_2}) = \one_{\ell_1\cdot\ell_2}
\]
for any $\ell_1\in\rL[S_1]$ and $\ell_2\in\rL[S_2]$. Similarly,
we see from~\eqref{e:f-coprod2} that
\[
\Delta_{S_1,S_2}(\one_\ell) = \one_{\ell|_{S_1}}\otimes\one_{\ell|_{S_2}}
\]
for any $\ell\in\rL[I]$. We thus have:
\begin{corollary}\label{c:fromL}
The collection of maps
\[
\bL[I] \to \FC(\Un)[I], \quad \ell\mapsto \one_\ell
\]
is an injective morphism of Hopf monoids. 
\end{corollary}

\subsection{Class functions on unitriangular matrices}\label{ss:class-unit}

Let $\CF(\Un)[I]$ be the direct sum of the spaces of class functions on
all unitriangular groups on $I$:
\[
\CF(\Un)[I] = \bigoplus_{\ell\in\rL[I]} \CF\bigl(\Un(I,\ell)\bigr).
\]
This defines a subspecies $\CF(\Un)$ of $\FC(\Un)$.

Proceeding in the same manner
as in Section~\ref{ss:f-unit}, we obtain linear maps
\[
\xymatrix@C+20pt{
\CF(\Un)[S_1]\otimes \CF(\Un)[S_2] \ar@<0.5ex>[r]^-{\mu_{S_1,S_2}} & 
\CF(\Un)[I] \ar@<0.5ex>[l]^-{\Delta_{S_1,S_2}} 
}
\]
by applying the functor $\CF$ to the morphisms $\pi_{S_1,S_2}$ and $\sigma_{S_1,S_2}$.
This is meaningful since the latter are 
morphisms of groups (in the case of $\pi_{S_1,S_2}$ by Lemma~\ref{l:principal-unit}).

\begin{proposition}\label{p:class-unit}
With these operations,
the species $\CF(\Un)$ is a connected cocommutative Hopf monoid. It is a Hopf submonoid of $\FC(\Un)$.
\end{proposition}
\begin{proof}
As in the proof of Proposition~\ref{p:f-unit}, the first statement follows
by functoriality. The second follows from the naturality of the inclusion
of class functions and its compatibility with the isomorphisms in~\eqref{e:f-strong} and~\eqref{e:strong}.
\end{proof}

We describe the operations on the basis of characteristic functions~\eqref{e:char-class}. Let $C_i$ be a conjugacy class of $\Un(S_i,\ell_i)$, $i=1,2$.
It follows from~\eqref{e:class-functor} and~\eqref{e:class-strong} that the product is
\begin{equation}\label{e:char-class-prod}
\mu_{S_1,S_2}(\kappa_{C_1}\otimes \kappa_{C_2}) = \sum_{\pi_{S_1,S_2}(C)\subseteq C_1\times C_2} \kappa_C,
\end{equation}
where the sum is over  conjugacy classes $C$ in $\Un(I,\ell_1\cdot\ell_2)$.
Similarly, the coproduct is
\begin{equation}\label{e:char-class-coprod}
\Delta_{S_1,S_2}(\kappa_{C}) = \sum_{\sigma_{S_1,S_2}(C_1\times C_2)\subseteq C} \kappa_{C_1}\otimes \kappa_{C_2}.
\end{equation}
Here $C$ is a conjugacy class of $\Un(I,\ell)$ and the sum is over pairs of conjugacy
classes $C_i$ of $\Un(S_i,\ell|_{S_i})$.

\begin{remark}
Let 
\[
\Fc:\{\text{groups}\} \to \{\text{vector spaces}\}
\]
be a functor that is contravariant and \emph{bilax monoidal}, in the sense
of~\cite[Section~3.1]{AM:2010}.
The construction of the Hopf monoids $\FC(\Un)$ and $\CF(\Un)$ can be carried out for any
such functor $\Fc$ in place of $\CF$, in exactly the same manner.
It can also be carried out for a covariant bilax monoidal functor $\Fc$, in
a similar manner.
\end{remark}

\subsection{A combinatorial model}\label{ss:f-model}

To a unitriangular matrix $U\in \Un(I,\ell)$ we associate a graph $g(U)$ on $I$
as follows: there is an edge between $i$ and $j$ if $i<j$ in $\ell$ and $u_{ij}\neq 0$.
For example, given nonzero entries $a,b,c\in\Fb$,
\begin{equation}
\begin{gathered}
\ell=hijk,
\qquad
U=\pmat{1 & 0 & 0 & a \\
& 1 & b & c \\
& & 1 & 0\\
& & & 1}
\quad \Rightarrow \quad
g(U)=\xymatrix@R-25pt{
*{\bullet} \ar@{-}@/^3pc/[rrr] & 
*{\bullet} \ar@{-}@/^2pc/[rr]  \ar@{-}@/^1pc/[r] & 
*{\bullet} & *{\bullet} \\
 h & i & j & k
}.
\end{gathered}
\end{equation}

Recall the Hopf monoids $\bL$ and $\bG$ and the notion of Hadamard
product from Section~\ref{s:hopf}.
Let 
\[
\phi: \bL\times\bG \to \FC(\Un) 
\]
be the map with components 
\[
(\bL\times\bG)[I]  \to \FC(\Un)[I]
\]
given as follows. On a basis element
$\ell\otimes m_g \in \bL[I]\otimes \bG[I]  = (\bL\times\bG)[I]$, we set
\begin{equation}\label{e:f-model}
\phi(\ell\otimes m_g) = \sum_{U\in\Un(I,\ell):\, g(U)=g} \kappa_U \in \FC\bigl(\Un(I,\ell)\bigr) \subseteq \FC(\Un)[I],
\end{equation}
and extend by linearity.
The map relates the $m$-basis of $\bG$ to the basis of characteristic
functions~\eqref{e:char-function} of $\FC(\Un)$.

\begin{proposition}\label{p:f-model}
Let $\Fb$ be an arbitrary finite field.
The map $\phi: \bL\times\bG \to \FC(\Un)$ is an  injective morphism of Hopf monoids.
\end{proposition}
\begin{proof}
From the definition of the Hopf monoid operations on a Hadamard product
and formulas~\eqref{e:L} and~\eqref{e:prod-g} it follows that
\[
\mu_{S_1,S_2}\bigl((\ell_1\otimes m_{g_1})\otimes (\ell_2\otimes m_{g_2})\bigr)
= \sum_{g|_{S_1}=g_1,\,g|_{S_2}=g_2} \ell_1\cdot\ell_2 \otimes m_g.
\]
Comparing with formula~\eqref{e:char-f-prod} we see that products
are preserved, since given $U\in\Un(I,\ell)$, we have
\[
g(U_{S_i}) = g(U)|_{S_i}.
\]
The verification for coproducts is similar, employing~\eqref{e:L},~\eqref{e:coprod-g}
and~\eqref{e:char-f-coprod} and the fact that given $I=S_1\sqcup S_2$ and $U_i\in\Un(S_i,\ell|_{S_i})$, we have
\[
g(U_1\oplus U_2) = g(U_1)\sqcup g(U_2).
\]

Consider the map $\psi: \FC(\Un)\to \bL\times\bG$ given by
\begin{equation}\label{e:phi-inv}
\psi(\kappa_U)= \ell\otimes m_{g(U)}
\end{equation} 
for any $U\in\Un(I,\ell)$.
Then 
\[
\psi\phi(\ell\otimes m_g) = (q-1)^{e(g)}\,\ell\otimes m_g,
\]
where $q$ is the cardinality of $\Fb$ and $e(g)$ is the number of edges in $g$.
Thus $\phi$ is injective.
\end{proof}

We mention that the map $\psi$ in~\eqref{e:phi-inv}
is a morphism of comonoids, but not of monoids in general.

Assume now that the matrix entries are from $\Fb_2$, the field
with $2$ elements. In this case, the matrix $U$ is uniquely determined
by the linear order $\ell$ and the graph $g(U)$. Therefore, the map
$\phi$ is invertible, with inverse $\psi$.

\begin{corollary}\label{c:f-model}
There is an isomorphism of Hopf monoids
\[
\FC(\Un) \cong \bL\times\bG
\]
between the Hopf monoid of functions on unitriangular matrices with 
entries in $\Fb_2$ and the Hadamard product of the Hopf monoids of linear
orders and simple graphs.
\end{corollary}

On an arbitrary function $f:\Un(I,\ell)\to\field$, the isomorphism is given by
\[
\psi(f)=  \ell\otimes \sum_{U\in\Un(I,\ell)} f(U)\,m_{g(U)}.
\]
The coefficients of the $m$-basis elements are the values of $f$.

\section{A Hopf monoid of superclass functions}\label{s:hopf-super}

An abstract notion of \emph{superclass} (and \emph{supercharacter})
has been introduced
by Diaconis and Isaacs~\cite{DI:2008}.
We only need a minimal amount of related concepts that we review
in Sections~\ref{ss:super-group} and~\ref{ss:super-unit}.
For this purpose we first place ourselves in the setting of algebra groups.
In Section~\ref{ss:super-unit} we
construct a Hopf monoid structure on the collection of spaces
of superclass functions on the unitriangular groups, by the same
procedure as that in Section~\ref{s:hopf-class}. The combinatorics of these
superclasses is understood from Yan's thesis~\cite{Yan:2001} (reviewed in
slightly different terms in Section~\ref{ss:super-class}), and this allows us to obtain an explicit 
description for the Hopf monoid operations in Section~\ref{ss:super-prod}.
This leads to a theorem in Section~\ref{ss:super-structure} identifying the Hopf monoid of superclass functions with matrix entries
in $\Fb_2$ to the Hadamard product of the Hopf monoids of linear orders and set partitions. The combinatorial models for functions and for superclass functions
are related in Section~\ref{ss:rel-model}.

\subsection{Superclass functions on algebra groups}\label{ss:super-group}

Let $\Nf$ be a \emph{nilpotent} algebra: an associative, nonunital algebra in which every element
is nilpotent. 
Let $\bNf=\Fb\oplus\Nf$ denote the result of adjoining a unit to $\Nf$.
The set
\[
G(\Nf) = \{1+n \mid n\in\Nf\}
\]
is a subgroup of the group of invertible elements of $\bNf$.
A group of this form is called an \emph{algebra group}.
(This is the terminology employed in~\cite{DI:2008},
and in a slightly different context, in~\cite{And:1999} and~\cite{Isa:1995}.)

A morphism of nilpotent algebras $\varphi:\Mf\to\Nf$ has a unique unital extension
$\bMf\to\bNf$ and this sends $G(\Mf)$ to $G(\Nf)$. A \emph{morphism of algebra groups} is a map of this form. 

\begin{warning}
When we refer to the algebra group $G(\Nf)$, it is implicitly assumed that the
algebra $\Nf$ is given as well.
\end{warning}

Following Yan~\cite{Yan:2001}, we define an equivalence relation on $G(\Nf)$ as follows.
Given $x$ and $y\in G(\Nf)$, we write $x\sim y$ if there exist $g$ and $h\in G(\Nf)$ such that
\begin{equation}\label{e:super-rel}
y-1 = g(x-1)h.
\end{equation}
Following now Diaconis and Isaacs~\cite{DI:2008}, we refer to the equivalence classes of this relation as \emph{superclasses} and to the functions $G(\Nf) \to \field$ constant on
these classes as \emph{superclass functions}. The set of such functions is
denoted $\SC\bigl(G(\Nf)\bigr)$. 

Since
\[
gxg^{-1}-1 = g(x-1)g^{-1},
\]
we have that $x\sim gxg^{-1}$ for any $x$ and $g\in G(\Nf)$. Thus,
each superclass is a union of conjugacy classes, and hence every superclass function is a class function:
\begin{equation}\label{e:super-conj}
\SC\bigl(G(\Nf)\bigr) \subseteq \CF\bigl(G(\Nf)\bigr).
\end{equation}

A morphism $\varphi: G(\Mf)\to G(\Nf)$ of algebra groups
preserves the relation $\sim$. Therefore, if $f:G(\Nf) \to \field$
is a superclass function on $G(\Nf)$, then $f\circ \varphi$ is a superclass function
on $G(\Mf)$. In this manner,
\[
\SC: \{\text{algebra groups}\} \to \{\text{vector spaces}\}
\]
is a contravariant functor. In addition, the inclusion~\eqref{e:super-conj} is natural
with respect to morphisms of algebra groups.

The direct product of two algebra groups is another algebra group. Indeed,
\[
G(\Nf_1)\times G(\Nf_2) \cong G(\Nf_1\oplus\Nf_2)
\]
and $\Nf_1\oplus\Nf_2$ is nilpotent. Moreover,
\[
(x_1,x_2) \sim (y_1,y_2) \iff (x_1\sim y_1 \text{ and } x_2\sim y_2).
\]
Therefore, a superclass of the product is a pair of superclasses
from the factors, and if at least one of the two groups is finite, there is a canonical isomorphism
\[
\SC\bigl(G(\Nf_1)\times G(\Nf_2)\bigr) \cong 
\SC\bigl(G(\Nf_1)\bigr) \otimes \SC\bigl(G(\Nf_2)\bigr).
\]

\subsection{Superclass functions on unitriangular matrices}\label{ss:super-unit}

Given a finite set $I$ and a linear order $\ell$ on $I$, let $\Nf(I,\ell)$ denote
the subalgebra of $\Mat(I)$ consisting of strictly upper triangular matrices
\[
N=(n_{ij})_{i,j\in I},\ n_{ij}=0 \text{ whenever } i\geq_{\ell} j.
\]
Then $\Nf(I,\ell)$ is nilpotent and $G\bigl(\Nf(I,\ell)\bigr)=\Un(I,\ell)$.
Thus, the unitriangular groups are algebra groups.

We assume from now on that the field $\Fb$ is finite. 

We define, for each finite set $I$,
\[
\SC(\Un)[I] = \bigoplus_{\ell\in\rL[I]} \SC\bigl(\Un(I,\ell)\bigr).
\]
This defines a species $\SC(\Un)$. Proceeding in the same manner
as in Sections~\ref{ss:f-unit} and~\ref{ss:class-unit}, we obtain linear maps
\[
\xymatrix@C+20pt{
\SC(\Un)[S_1]\otimes \SC(\Un)[S_2] \ar@<0.5ex>[r]^-{\mu_{S_1,S_2}} & 
\SC(\Un)[I] \ar@<0.5ex>[l]^-{\Delta_{S_1,S_2}} 
}
\]
by applying the functor $\SC$ to the morphisms $\pi_{S_1,S_2}$ and $\sigma_{S_1,S_2}$.
This is meaningful since the latter are 
morphisms of algebra groups: it was noted in Section~\ref{ss:direct}
that $\sigma_{S_1,S_2}$ is the restriction of a morphism defined on the full
matrix algebras, while the considerations of Lemma~\ref{l:principal-unit}
show that $\pi_{S_1,S_2}$ is the restriction of a morphism defined on
the algebra of upper triangular matrices.

\begin{proposition}\label{p:super-unit}
With these operations,
the species $\SC(\Un)$ is a connected cocommutative Hopf monoid.
It is a Hopf submonoid of $\CF(\Un)$.
\end{proposition}
\begin{proof}
As in the proof of Proposition~\ref{p:f-unit}, the first statement follows
by functoriality. The second follows from the naturality of the inclusion~\eqref{e:super-conj}.
\end{proof}

Formulas~\eqref{e:char-class-prod} and~\eqref{e:char-class-coprod} continue to hold
for the (co)product of superclass functions. 

\smallskip

The constant function $\one_\ell$ is a superclass function. Thus the morphism of
Hopf monoids of Corollary~\ref{c:fromL} factors through $\SC(\Un)$ and $\CF(\Un)$:
\[
\bL \into \SC(\Un) \into \CF(\Un) \into \FC(\Un).
\]

\subsection{Combinatorics of the superclasses}\label{ss:super-class}

Yan~\cite{Yan:2001} showed that superclasses are parametrized by certain combinatorial
data, essentially along the lines presented below.

According to~\eqref{e:super-rel}, two unitriangular matrices $U_1$ and $U_2$ are
in the same superclass if and only if $U_2-\Id$ is obtained from $U_1-\Id$ by a sequence
of elementary row and column operations. The available operations are from the
unitriangular group itself, so the pivot entries cannot be normalized.
Thus, each superclass contains a unique matrix $U$
such that $U-\Id$ has at most one nonzero
entry in each row and each column. We refer to this matrix $U$ as
the \emph{canonical representative} of the superclass.

We proceed to encode such representatives in terms of combinatorial data.

We first discuss the combinatorial data.
Let $\ell$ be a linear order on a finite set $I$ and $X$ a partition of $I$.
Let us say that $i$ and $j\in I$ \emph{bound an arc} if:
\begin{itemize}
\item $i$ precedes $j$ in $\ell$;
\item $i$ and $j$ are in the same block of $X$, say $S$;
\item no other element of $S$ lies between $i$ and $j$ in the order $\ell$.
\end{itemize}
The set of \emph{arcs} is
\[
A(X,\ell):=\{(i,j)  \mid \text{ $i$ and $j$ bound an arc}\}.
\]
Consider also a function 
\[
\alpha: A(X,\ell) \to \Fb^{\times}
\]
from the set of arcs to the nonzero elements of $\Fb$. 
We say that the pair $(X,\alpha)$ is an \emph{arc diagram} on the linearly ordered set $(I,\ell)$.
We may visualize an arc diagram as follows.
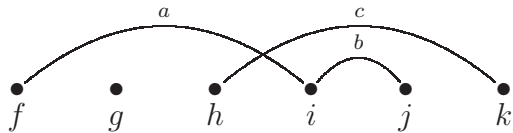
\[
\xymatrix@R-25pt{
*{\bullet} \ar@{-}@/^2pc/[rrr]^-{a} & *{\bullet} &
*{\bullet} \ar@{-}@/^2pc/[rrr]^-{c} & *{\bullet} \ar@{-}@/^1pc/[r]^-{b}
& *{\bullet} & *{\bullet} \\
f & g & h & i & j & k
}
\]
Here the combinatorial data is
\[
\ell=fghijk,\quad X=\bigl\{ \{f,i,j\},\{g\},\{h,k\} \bigr\},
\quad \alpha(f,i)=a,\ \alpha(i,j)=b,\ \alpha(h,k)=c.
\]

Fix the linear order $\ell$. To an arc diagram $(X,\alpha)$ on $(I,\ell)$
we associate a matrix $U_{X,\alpha}$ with entries
\[
u_{ij} = 
\begin{cases}
\alpha(i,j) & \text{ if } (i,j)\in A(X,\ell), \\
 1        & \text{ if } i=j,\\
 0 & \text{ otherwise.}
\end{cases}
\]
Clearly, the matrix $U_{X,\alpha}$ is $\ell$-unitriangular and 
$U_{X,\alpha}-\Id$ has at most one nonzero
entry in each row and each column. In the above example,
\[
U_{X,\alpha} = 
\pmat{1 & 0 & 0 & a & 0 & 0\\
& 1 & 0 & 0 & 0 & 0\\
& & 1 & 0 & 0 & c\\
& & & 1 & d & 0 \\
& & & & 1 & 0 \\
& & & & & 1
}.
\]

Conversely, any canonical representative matrix $U\in\Un(I,\ell)$ is of the form $U_{X,\alpha}$ for a unique arc diagram $(X,\alpha)$ on $(I,\ell)$: the location of the nonzero
entries determines the set of arcs and the values of the entries determine
the function $\alpha$. The smallest equivalence relation on $I$ containing
the set of arcs determines the partition $X$. 

In conclusion, the canonical representatives, and hence the superclasses,
are in bijection with the set of arc diagrams. We let $C_{X,\alpha}$
denote the superclass of $\Un(I,\ell)$ containing $U_{X,\alpha}$ and we write
$\kappa_{X,\alpha}$ for the characteristic function of this class.
As $(X,\alpha)$ runs over all arc diagrams on $(I,\ell)$, these functions
form a basis of the space $\SC\bigl(\Un(I,\ell)\bigr)$.

We describe principal minors and direct sums of the canonical representatives.
To this end, fix $\ell\in\rL[I]$, and recall the notions of union and restriction of set partitions discussed
in Section~\ref{ss:partition}.

Let $S\subseteq I$ be an arbitrary subset. Given a partition $X$
of $I$, let $A(X,\ell)|_S$ denote the subset of $A(X,\ell)$
consisting of those arcs $(i,j)$ where both $i$ and $j$ belong to $S$.
We let $\alpha|_S$ denote the restriction of $\alpha$ to $A(X,\ell)|_S$.
We have $A(X,\ell)|_S \subseteq A(X|_S,\ell|_S)$, and if $S$ is an $\ell$-segment, then
\begin{equation}\label{e:arc-restriction}
A(X,\ell)|_S = A(X|_S,\ell|_S).
\end{equation}
In this case, we obtain an arc diagram $(X|_S,\alpha|_S)$ on $(S,\ell|_S)$, and
we have
\begin{equation}\label{e:rep-principal}
(U_{X,\alpha})_S = U_{X|_S,\alpha|_S}.
\end{equation}

Suppose now that $I=S_1\sqcup S_2$ and $(X_i,\alpha_i)$ 
is an arc diagram on $(S_i,\ell|_{S_i})$, $i=1,2$. Then
\begin{equation}\label{e:arc-union}
A(X_1\sqcup X_2,\ell) = A(X_1,\ell|_{S_1})\sqcup A(X_2,\ell|_{S_2}).
\end{equation}
Let $\alpha_1\sqcup\alpha_2$ denote the common extension of $\alpha_1$
and $\alpha_2$ to this set. Then $(X_1\sqcup X_2,\alpha_1\sqcup\alpha_2)$
is an arc diagram on $(I,\ell)$ and we have
\begin{equation}\label{e:rep-direct}
U_{X_1,\alpha_1}\oplus U_{X_2,\alpha_2} = U_{X_1\sqcup X_2,\alpha_1\sqcup\alpha_2}.
\end{equation}

\subsection{Combinatorics of the (co)product}\label{ss:super-prod}

We now describe the product and coproduct of the Hopf monoid $\SC(\Un)$
on the basis $\{\kappa_{X,\alpha}\}$ of Section~\ref{ss:super-class}.
We employ formulas~\eqref{e:char-class-prod} and~\eqref{e:char-class-coprod}
which as discussed in Section~\ref{ss:super-unit} hold for superclass functions.

Let $I=S_1\sqcup S_2$ and $\ell_i\in\rL[S_i]$, $i=1,2$, and consider
the product
\[
\SC\bigl(\Un(S_1,\ell_1)\bigr)\times \SC\bigl(\Un(S_2,\ell_2)\bigr)
\to \SC\bigl(\Un(I,\ell_1\cdot\ell_2)\bigr).
\]
Let $(X_i,\alpha_i)$ be an arc diagram on $(I,\ell_i)$, $i=1,2$.
According to~\eqref{e:char-class-prod}
we have
\[
\mu_{S_1,S_2}(\kappa_{X_1,\alpha_1}\otimes \kappa_{X_2,\alpha_2}) = \sum_{\pi_{S_1,S_2}(C_{X,\alpha})\subseteq C_{X_1,\alpha_1}\times C_{X_2,\alpha_2}} \kappa_{X,\alpha},
\]
a sum over arc diagrams $(X,\alpha)$ on $(I,\ell_1\cdot\ell_2)$.
Since $\pi_{S,T}$ preserves superclasses, 
\begin{align*}
\pi_{S_1,S_2}(C_{X,\alpha})\subseteq C_{X_1,\alpha_1}\times C_{X_2,\alpha_2}
& \iff
\pi_{S_1,S_2}(U_{X,\alpha})\in C_{X_1,\alpha_1}\times C_{X_2,\alpha_2}\\
& \iff (U_{X,\alpha})_{S_i} \in C_{X_i,\alpha_i},\ i=1,2.
\end{align*}
In view of~\eqref{e:rep-principal}, this is in turn equivalent to
\[
X|_{S_i} = X_i \qand \alpha|_{S_i}, = \alpha_i, \quad i=1,2.
\]
In conclusion,
\begin{equation}\label{e:char-class-prod2}
\mu_{S_1,S_2}(\kappa_{X_1,\alpha_1}\otimes \kappa_{X_2,\alpha_2}) = 
\sum_{X|_{S_i} = X_i,\, \alpha|_{S_i} = \alpha_i} 
\kappa_{X,\alpha}.
\end{equation}
The sum is over all arc diagrams $(X,\alpha)$ on $(I,\ell_1\cdot\ell_2)$ whose restriction to $S_i$
is $(X_i,\alpha_i)$ for $i=1,2$.

Take now $\ell\in\rL[I]$, $I=S_1\sqcup S_2$, and consider the coproduct
\[
 \SC\bigl(\Un(I,\ell)\bigr) \to  \SC\bigl(\Un(S_1,\ell|_{S_1})\bigr) \times  \SC\bigl(\Un(S_2,\ell|_{S_2})\bigr).
\]
Let $(X,\alpha)$ be an arc diagram on $(I,\ell)$. According to~\eqref{e:char-class-coprod} we have
\[
\Delta_{S_1,S_2}(\kappa_{X,\alpha}) = \sum_{\sigma_{S_1,S_2}(C_{X_1,\alpha_1}\times C_{X_2,\alpha_2})\subseteq C_{X,\alpha}} \kappa_{X_1,\alpha_1}\otimes \kappa_{X_2,\alpha_2},
\]
a sum over arc diagrams $(X_i,\alpha_i)$ on $(S_i,\ell|_{S_i})$.
The superclass $C_{X_1,\alpha_1}\times C_{X_2,\alpha_2}$ contains
$
(U_{X_1,\alpha_1},\,U_{X_2,\alpha_2})
$
and hence its image under $\sigma_{S_1,S_2}$ contains
\[
U_{X_1,\alpha_1}\oplus U_{X_2,\alpha_2} = U_{X_1\sqcup X_2,\alpha_1\sqcup\alpha_2},
\]
by~\eqref{e:rep-direct}. Therefore,
\[
\sigma_{S_1,S_2}(C_{X_1,\alpha_1}\times C_{X_2,\alpha_2})\subseteq C_{X,\alpha}
\iff
X_1\sqcup X_2 = X \qand \alpha_1\sqcup\alpha_2 = \alpha.
\]
Note that $X_1\sqcup X_2 = X$ if and only if $S_1$ (or equivalently, $S_2$)
is a union of blocks of $X$. In this case, $X_i=X|_{S_i}$ and
$\alpha_i=\alpha|_{S_i}$.
In conclusion,
\begin{equation}\label{e:char-class-coprod2}
\Delta_{S_1,S_2}(\kappa_{X,\alpha}) =
\begin{cases}
\kappa_{X|_{S_1},\alpha|_{S_1}} \otimes \kappa_{X|_{S_2},\alpha|_{S_2}}  & \text{if $S_1$ is the union of some blocks of $X$,} \\
0 & \text{otherwise.}
\end{cases}
\end{equation}

\subsection{Decomposition as a Hadamard product}\label{ss:super-structure}

The apparent similarity between the combinatorial description of the
Hopf monoid operations of $\SC(\Un)$ in Section~\ref{ss:super-prod}
and those of the Hopf monoids $\bL$ and $\bPi$ in Sections~\ref{ss:order}
and~\ref{ss:partition} can be formalized. Recall the Hadamard product
of Hopf monoids from Section~\ref{ss:hadamard}.

Let 
\[
\phi: \bL\times\bPi \to \SC(\Un) 
\]
be the map with components 
\[
(\bL\times\bPi)[I]  \to \SC(\Un)[I]
\]
given as follows. On a basis element
$\ell\otimes m_X \in \bL[I]\otimes \bPi[I]  = (\bL\times\bPi)[I]$, we set
\begin{equation}\label{e:s-model}
\phi(\ell\otimes m_X) = \sum_{\alpha:A(X,\ell)\to \Fb^{\times}} \kappa_{X,\alpha} \in \SC\bigl(\Un(I,\ell)\bigr) \subseteq \SC(\Un)[I],
\end{equation}
and extend by linearity. The morphism $\phi$ adds labels to the arcs 
in all possible ways.

\begin{proposition}\label{p:s-model}
Let $\Fb$ be an arbitrary finite field.
The map $\phi: \bL\times\bPi \to \SC(\Un)$ is an  injective morphism of Hopf monoids.
\end{proposition}
\begin{proof}
This follows by comparing definitions, as in the proof of Proposition~\ref{p:f-model}.
The relevant formulas are~\eqref{e:L}, \eqref{e:prod-m} and~\eqref{e:coprod-m}
for the operations of $\bL\times\bPi$, and~\eqref{e:char-class-prod2} and~\eqref{e:char-class-coprod2} for the operations of $\SC(\Un)$.
\end{proof}

When the field of matrix entries is $\Fb_2$, the arc labels are uniquely determined.
The map $\phi$ is then invertible, with inverse $\psi$ given by
\[
\psi(\kappa_{X,\alpha}) = \ell\otimes m_X
\]
for any arc diagram $(X,\alpha)$ on a linearly ordered set $(I,\ell)$.
We thus have the following.

\begin{corollary}\label{c:super-structure}
There is an isomorphism of Hopf monoids
\[
\SC(\Un) \cong \bL\times\bPi
\]
between the Hopf monoid of superclass functions on unitriangular matrices with 
entries in $\Fb_2$ and the Hadamard product of the Hopf monoids of linear
orders and set partitions.
\end{corollary}

\subsection{Relating the combinatorial models}\label{ss:rel-model}

The results of Section~\ref{ss:super-structure} provide a combinatorial
model for the Hopf monoid $\SC(\Un)$. They 
parallel those of Section~\ref{ss:f-model} which do the same for $\FC(\Un)$.
We now interpret the inclusion $\SC(\Un)\into\FC(\Un)$ in these terms.

Let $X$ be a partition on a linearly ordered set $(I,\ell)$.
We may regard the set of arcs $A(X,\ell)$ as a simple graph on $I$.
Let $G(X,\ell)$ denote the set of simple graphs $g$ on $I$ such that:
\begin{itemize}
\item $g$ contains the graph $A(X,\ell)$;
\item if $i<j$ in $\ell$ and $g\setminus A(X,\ell)$ contains an edge between $i$ and $j$, then there exists $k$ such that
\[
\text{$i<k<j$ in $\ell$ and either $(i,k)\in A(X,\ell)$ or $(k,j)\in A(X,\ell)$.}
\]
\end{itemize}
The following illustrates the extra edges (dotted) that may be present 
in $g$ when an arc (solid) is present in $A(X,\ell)$.
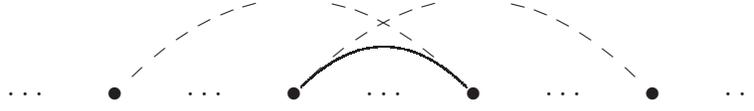
\begin{equation*}
\begin{gathered}
\xymatrix@R+10pt{
\\
*{\ldots} & *{\bullet} \ar@{--}@/^3pc/[rrrr] & *{\ldots}  & *{\bullet}\ar@{-}@/^1.5pc/[rr] 
\ar@{--}@/^3pc/[rrrr] & 
*{\ldots} & *{\bullet}  & *{\ldots}  & *{\bullet} & *{\ldots} 
}
\end{gathered}
\end{equation*}

Define a map
\[
\bL\times\bPi \to \bL\times\bG
\]
with components 
\[
(\bL\times\bPi)[I]  \to (\bL\times\bG)[I],
\quad
\ell\otimes m_X \mapsto \ell\otimes \sum_{g\in G(X,\ell)} m_g.
\]

\begin{proposition}\label{p:rel-model}
The map $\bL\times\bPi \to \bL\times\bG$ is an injective morphism of Hopf monoids.
Moreover, the following diagram commutes.
\[
\xymatrix{
\bL\times\bPi \xyinc[r]  \ar[d]_{\phi}  & \bL\times\bG  \ar[d]^{\phi}  \\
\SC(\Un)  \xyinc[r]  & \FC(\Un)
}
\]
\end{proposition}
\begin{proof}
It is enough to prove the commutativity of the diagram, since all other maps
in the diagram are injective morphisms. The commutativity boils down to the following fact.
Given $X\in\Pi[I]$, $\ell\in\rL[I]$, and $U\in\Un(I,\ell)$,
\begin{center}
$U\in C_{X,\alpha}$ for some $\alpha:A(X,\ell)\to\Fb^{\times}$
$\iff$
$g(U) \in G(X,\ell)$.
\end{center}
This expresses the fact that a matrix $U$ belongs to the superclass $C_{X,\alpha}$ if and only if
the nonzero entries of $U-\Id$ are located either above or to the right
of the nonzero entries of the representative $U_{X,\alpha}$.
\end{proof}

\section{Freeness}\label{s:free}

We prove that the Hopf monoids $\FC(\Un)$ and $\SC(\Un)$
are free, and the Hopf structure is isomorphic to the canonical
one on a free monoid. We assume that the base field $\field$ is of characteristic $0$,
which enables us to apply the results of the appendix.

\subsection{A partial order on arc diagrams}\label{ss:partial}

Let $(I,\ell)$ be a linearly ordered set. Given arc diagrams $(X,\alpha)$
and $(Y,\beta)$ on $(I,\ell)$, we write
\[
(X,\alpha)\leq (Y,\beta)
\]
if
\[
A(X,\ell) \subseteq A(Y,\ell)
\qand
\beta|_{A(X,\ell)} = \alpha.
\]
In other words, every arc of $X$ is an arc of $Y$, and with the same label. In particular, the partition $Y$ is coarser than $X$. On the other hand, the
following arc diagrams are incomparable (regardless of the labels), even though
the partition on the right is the coarsest one.
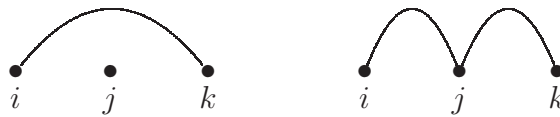
\vspace*{5pt}
\[
\xymatrix@R-25pt{
\\
\\
*{\bullet} \ar@{-}@/^2pc/[rr] & *{\bullet}  & *{\bullet} \\
 i & j & k
}
\qquad\qquad
\xymatrix@R-25pt{
\\
\\
*{\bullet} \ar@{-}@/^2pc/[r] &
*{\bullet} \ar@{-}@/^2pc/[r] & 
*{\bullet} \\
i & j & k
}
\]

The poset of arc diagrams has a unique minimum (the partition into singletons,
for which there are no arcs), but several maximal elements. The arc
diagrams above are the two maximal elements when $\ell=ijk$
(up to a choice of labels).

A partition $X$ of the linearly ordered set $(I,\ell)$ is \emph{atomic} if no proper initial 
$\ell$-segment of $I$ is a union of blocks of $X$. Equivalently, there
is no decomposition $I=S_1\sqcup S_2$ into proper $\ell$-segments such that
$X=X|_{S_1}\sqcup X|_{S_2}$.  
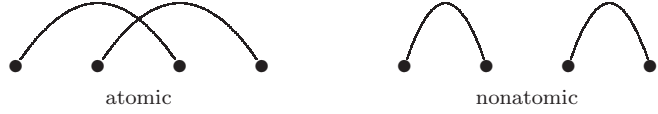
\vspace*{15pt}
\[
\xymatrix@R-25pt{
*{\bullet} \ar@{-}@/^2pc/[rr] & *{\bullet} \ar@{-}@/^2pc/[rr] 
 & *{\bullet} & *{\bullet}  \\
 &  \ar@{}[r]_{\text{atomic}} & &
}
\qquad\qquad
\xymatrix@R-25pt{
*{\bullet} \ar@{-}@/^2pc/[r] & *{\bullet}   
& *{\bullet} \ar@{-}@/^2pc/[r]  & *{\bullet} \\
&  \ar@{}[r]_{\text{nonatomic}} & &
}
\]

If $(X,\alpha)$ is a maximal element of the poset of arc diagrams,
then $X$ is an atomic partition. But if $X$ is atomic, $(X,\alpha)$
need not be maximal (regardless of $\alpha$).
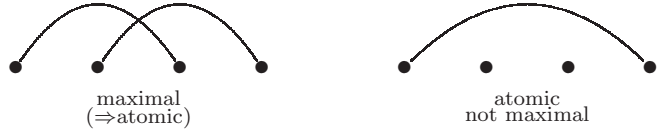
\vspace*{15pt}
\[
\xymatrix@R-25pt{
*{\bullet} \ar@{-}@/^2pc/[rr] & *{\bullet} \ar@{-}@/^2pc/[rr] 
 & *{\bullet} & *{\bullet}  \\
 &  \ar@{}[r]_{\text{maximal}} & & \\
  &  \ar@{}[r]_{(\Rightarrow \text{atomic})} & &
}
\qquad\qquad
\xymatrix@R-25pt{
*{\bullet} \ar@{-}@/^2pc/[rrr] & *{\bullet}   
& *{\bullet}  & *{\bullet} \\
&  \ar@{}[r]_{\text{atomic}} & & \\
  &  \ar@{}[r]_{\text{not maximal}} & &
}
\]

\subsection{A second basis for $\SC(\Un)$}\label{ss:lambda}

We employ the partial order of Section~\ref{ss:partial}
to define a new basis $\{\lambda_{X,\alpha}\}$
of $\SC\bigl(\Un(I,\ell)\bigr)$
by
\begin{equation}\label{e:lambda-basis}
\lambda_{X,\alpha} = \sum_{(X,\alpha)\leq (Y,\beta)} \kappa_{Y,\beta}.
\end{equation}

The product of the Hopf monoid
 $\SC(\Un)$ takes a simple form on the $\lambda$-basis.

\begin{proposition}\label{p:lambda-prod}
Let $I=S_1\sqcup S_2$ and $\ell_i\in \rL[S_i]$, $i=1,2$. Then
\begin{equation}\label{e:lambda-prod}
\mu_{S_1,S_2}(\lambda_{X_1,\alpha_1}\otimes\lambda_{X_2,\alpha_2})
= \lambda_{X_1\sqcup X_2,\alpha_1\sqcup\alpha_2}
\end{equation}
for any arc diagrams $(X_i,\alpha_i)$ on $(S_i,\ell_i)$, $i=1,2$.
\end{proposition}
\begin{proof}
We calculate using~\eqref{e:char-class-prod2} and~\eqref{e:lambda-basis}:
\[
\mu_{S_1,S_2}(\lambda_{X_1,\alpha_1}\otimes\lambda_{X_2,\alpha_2})=
\sum_{(X_i,\alpha_i)\leq (Y_i,\beta_i)} 
\mu_{S_1,S_2}(\kappa_{Y_1,\beta_1}\otimes\kappa_{Y_2,\beta_2})=
\sum_{(X_i,\alpha_i)\leq (Y|_{S_i},\beta|_{S_i})} \kappa_{Y,\beta}.
\]
Now, by~\eqref{e:arc-restriction} and~\eqref{e:arc-union} we have
\[
(X_i,\alpha_i)\leq (Y|_{S_i},\beta|_{S_i}),\ i=1,2\ \iff\ (X_1\sqcup X_2,\alpha_1\sqcup\alpha_2) \leq (Y,\beta).
\]
Therefore,
\[
\mu_{S_1,S_2}(\lambda_{X_1,\alpha_1}\otimes\lambda_{X_2,\alpha_2})=
\sum_{(X_1\sqcup X_2,\alpha_1\sqcup\alpha_2) \leq (Y,\beta)} \kappa_{Y,\beta}
= \lambda_{X_1\sqcup X_2,\alpha_1\sqcup\alpha_2}. \qedhere
\]
\end{proof}

The coproduct of the Hopf monoid
 $\SC(\Un)$ takes the same form on the $\lambda$-basis
 as on the $\kappa$-basis.

\begin{proposition}\label{p:lambda-coprod}
Let $I=S_1\sqcup S_2$ and $\ell\in\rL[I]$.
\begin{equation}\label{e:lambda-coprod}
\Delta_{S_1,S_2}(\lambda_{X,\alpha}) =
\begin{cases}
\lambda_{X|_{S_1},\alpha|_{S_1}} \otimes \lambda_{X|_{S_2},\alpha|_{S_2}}  & \text{if $S_1$ is the union of some blocks of $X$,} \\
0 & \text{otherwise.}
\end{cases}
\end{equation}
\end{proposition}
\begin{proof}
Suppose first that $S_1$ is not the union of blocks of $X$.
Then the same is true for any partition coarser than $X$; in particular,
for any partition $Y$ entering in~\eqref{e:lambda-basis}. In view of
~\eqref{e:char-class-coprod2}, we then have $\Delta_{S_1,S_2}(\lambda_{X,\alpha}) = 0$.

Otherwise, $X=X|_{S_1}\sqcup X|_{S_2}$ and $\alpha=\alpha|_{S_1}\sqcup \alpha|_{S_2}$.
 Among the arc diagrams $(Y,\beta)$
entering in~\eqref{e:lambda-basis}, only those for which $Y=Y|_{S_1}\sqcup Y|_{S_2}$ contribute to the coproduct, in view of~\eqref{e:char-class-coprod2}.
These arc diagrams are of the form $Y=Y_1\sqcup Y_2$, $\beta=\beta_1\sqcup\beta_2$,  and by~\eqref{e:arc-union}
we must have
\[
A(X|_{S_i},\ell|_{S_i})\subseteq A(Y_i,\ell|_{S_i}),
\quad
\beta_i|_{A(X|_{S_i},\ell|_{S_i})} = \alpha|_{S_i},
\quad
i=1,2.
\]
We then have
\[
\Delta_{S_1,S_2}(\lambda_{X,\alpha}) =
\sum_{(X,\alpha)\leq (Y,\beta)} \Delta(\kappa_{Y,\beta}) =
\sum_{(X|_{S_i},\alpha|_{S_i})\leq (Y_i,\beta_i)} \kappa_{Y_1,\beta_1}\otimes \kappa_{Y_2,\beta_2}=  \lambda_{X|_{S_1},\alpha|_{S_1}} \otimes \lambda_{X|_{S_2},\alpha|_{S_2}}.
\]

\end{proof}

\begin{remark}
The relationship between the $\lambda$ and $\kappa$-bases of $\SC(\Un)$
is somewhat reminiscent of that between the $p$ and $m$-bases of $\bPi$~\cite[Equation~(12.5)]{AM:2010}. However, the latter involves a sum over all
partitions coarser than a given one. For this reason, the morphism $\phi$ in~\eqref{e:s-model},
which relates the $m$-basis to the $\kappa$-basis, does not relate the
$p$-basis to the $\lambda$-basis in the same manner.
\end{remark}

\subsection{Freeness of $\SC(\Un)$}\label{ss:free-sc}

Let $\bq$ be a species such that $\bq[\emptyset]=0$. A new species $\Tc(\bq)$
is defined by $\Tc(\bq)[\emptyset]=\field$ and, on a finite nonempty set $I$,
\[
\Tc(\bq)[I] = \bigoplus_{\substack{I=I_1\sqcup \cdots \sqcup I_k\\k\geq 1,\, I_j\neq\emptyset}} 
\bq[I_1]\otimes\cdots\otimes\bq[I_k].
\]
The sum is over all decompositions of $I$ into nonempty subsets. The number $k$ of
subsets is therefore bounded above by $\abs{I}$.

The species $\Tc(\bq)$ is a connected monoid with product given by
concatenation. To describe this in more detail,
let $I=S\sqcup T$, choose decompositions $S=S_1\sqcup\cdots\sqcup S_k$,
$T=T_1\sqcup\cdots\sqcup T_l$, and elements $x_i\in\bq[S_i]$, $i=1,\ldots,k$,
and $y_j\in\bq[T_j]$, $j=1,\ldots,l$. Write
\[
x=x_1\otimes\cdots\otimes x_i\in \Tc(\bq)[S]
\qand
y=y_1\otimes\cdots\otimes y_j\in \Tc(\bq)[T].
\]
The product is
\[
\mu_{S,T}(x\otimes y) = x_1\otimes\cdots\otimes x_i\otimes y_1\otimes\cdots\otimes y_j
\in \bq[S_1]\otimes\cdots\otimes\bq[S_k]\otimes\bq[T_1]\otimes\cdots\otimes\bq[T_l]
\subseteq \Tc(\bq)[I].
\]

The monoid $\Tc(\bq)$ is \emph{free} on the species $\bq$: 
a map of species $\bq\to\bmm$ from $\bq$ to a monoid $\bmm$ has a unique extension
to a morphism of monoids $\Tc(\bq)\to\bmm$.

The monoid $\Tc(\bq)$ may carry several coproducts that turn it into a 
connected Hopf monoid. 
The \emph{canonical} structure is the one for which the elements of $\bq$
are \emph{primitive}. This means that
\[
\Delta_{S,T}(x) = 0
\]
for every $x\in\bq[I]$ and every decomposition $I=S\sqcup T$ into nonempty subsets.

More details can be found in~\cite[Sections~11.2.1--11.2.2]{AM:2010}.

\smallskip

Let $\rD(I,\ell)$ denote the set of arc diagrams $(X,\alpha)$ for which
$X$ is an atomic set partition of the linearly ordered set $(I,\ell)$.
Let $\bd(I,\ell)$ be the vector space with basis $\rD(I,\ell)$.
Define a species $\bd$ by
\[
\bd[I] = \bigoplus_{\ell\in\rL[I]} \bd(I,\ell).
\]

Consider the map of species $\bd\to\SC$ with components
\[
\bd[I] \to \SC(\Un)[I], 
\quad
(X,\alpha) \mapsto \lambda_{X,\alpha}.
\]
The map sends the summand $\bd(I,\ell)$ of $\bd[I]$
to the summand $\SC\bigl(\Un(I,\ell)\bigr)$ of $\SC(\Un)[I]$.
By freeness, it extends to a morphism of monoids
\[
\Tc(\bd) \to \SC(\Un).
\]

\begin{proposition}\label{p:free-sc}
The map $\Tc(\bd) \to \SC(\Un)$ is an isomorphism of monoids.
In particular, the monoid $\SC(\Un)$ is free.
\end{proposition}
\begin{proof}
Let $(X,\alpha)$ be an arbitrary arc diagram on $(I,\ell)$.
Let $I_1,\ldots,I_k$ be the minimal $\ell$-segments of $I$, numbered from left to right,
such that each $I_j$ is a union of blocks of $X$. Let $\ell_j=\ell|_{I_j}$, $X_j=X|_{I_j}$, and
$\alpha_j=\alpha|_{I_j}$. Then $X_j$ is an atomic partition of $(I_j,\ell_j)$,
\[
X_1\sqcup \cdots \sqcup X_j = X
\qand
\alpha_1\sqcup \cdots \sqcup \alpha_j = \alpha.
\]
By~\eqref{e:lambda-prod},
\[
\mu_{I_1,\ldots,I_k}(\lambda_{X_1,\alpha_1}
\otimes\cdots\otimes\lambda_{X_k,\alpha_k}) = \lambda_{X,\alpha}.
\]
Thus, the morphism $\Tc(\bd) \to \SC(\Un)$ sends the basis element
$(X_1,\alpha_1)\otimes\cdots\otimes(X_k,\alpha_k)$ of
$\bd(I_1,\ell_1)\otimes\cdots\otimes\bd(I_k,\ell_k)$
to the basis element $\lambda_{X,\alpha}$ of $\SC\bigl(\Un(I,\ell)\bigr)$,
and is therefore an isomorphism.
\end{proof}

We may state Proposition~\ref{p:free-sc} by saying that the superclass functions
$\lambda_{X,\alpha}$ freely generate the monoid $\SC(\Un)$, as $(X,\alpha)$
runs over all arc diagrams for which $X$ is an atomic set partition.

The generators, however, need not be primitive. For instance,
\[
\xymatrix@R-25pt{
\\
\\
*{\bullet} \ar@{-}@/^2pc/[rr] & *{\bullet}  & *{\bullet} &  \ar@{|->}[rr]^-{\Delta_{\{i,k\},\{j\}}}
& & & *{\bullet} \ar@{-}@/^2pc/[r] &
*{\bullet} \ar@{}[r]^-{\otimes} & 
*{\bullet} \\
 i & j & k & &  & & i & k  & j
}
\]
which is not $0$.
Nevertheless, Proposition~\ref{p:freemon} allows us to conclude the following.

\begin{corollary}\label{c:free-sc}
Let $\field$ be a field of characteristic $0$. There exists an isomorphism
of Hopf monoids $\SC(\Un)\cong\Tc(\bd)$, where the latter is endowed
with its canonical Hopf structure.
\end{corollary}

As discussed in the appendix, an isomorphism may be constructed with the
aid of the first Eulerian idempotent.

Let $\bPi_a(I,\ell)$ be the vector space with basis the set of
atomic partitions on $(I,\ell)$.
When the field of matrix entries is $\Fb_2$,  arc diagrams reduce
to atomic set partitions and $\bd(I,\ell)$ identifies with $\bPi_a(I,\ell)$.
Combining Corollaries~\ref{c:super-structure} and~\ref{c:free-sc}
 we obtain an isomorphism
of Hopf monoids
\begin{equation}\label{e:free-partition}
\bL\times\bPi \cong \Tc(\bPi_a),
\end{equation}
where 
\[
\bPi_a[I]= \bigoplus_{\ell\in\rL[I]} \bPi_a(I,\ell).
\]

\subsection{A second basis for $\bG$ and for $\FC(\Un)$}\label{ss:lambda-graph}

Given two unitriangular matrices $U$ and $V\in\Un(I,\ell)$, let us write $U\leq V$
if
\[
u_{ij}=v_{ij} \text{ whenever }u_{ij}\neq 0.
\]
In other words, some zero entries in $U$ may be nonzero in $V$; the other entries
are the same for both matrices.

We define a new basis $\{\lambda_U\}$ of $\FC\bigl(\Un(I,\ell)\bigr)$
by
\[
\lambda_{U} = \sum_{U\leq V} \kappa_V.
\]

Let $I=S_1\sqcup S_2$, $U\in\Un(I,\ell)$, and $g_i\in\Un(S_i,\ell_i)$, $i=1,2$. 
It is easy to derive the following formulas from~\eqref{e:char-f-prod} and~\eqref{e:char-f-coprod}.
\begin{gather}\label{e:char-lambda-prod}
\mu_{S_1,S_2}(\lambda_{U_1}\otimes \lambda_{U_2}) = \lambda_{U_1\oplus U_2},
\\
\label{e:char-lambda-coprod}
\Delta_{S_1,S_2}(\lambda_{U}) = 
\begin{cases}
\lambda_{U_{S_1}}\otimes \lambda_{U_{S_2}} & \text{ if } U= U_{S_1} \oplus U_{S_2}, \\
   0      & \text{ otherwise.}
\end{cases}
\end{gather}

Formula~\eqref{e:char-lambda-prod} implies that $\FC(\Un)$ is a free monoid with generators $\lambda_U$
indexed by unitriangular matrices $U$ for which the graph $g(U)$ is connected.

For completeness, one may define a new basis $\{p_g\}$ of $\bG[I]$ by
\begin{equation}\label{e:p-graph}
p_g = \sum_{g\subseteq h} m_h.
\end{equation}
The sum is over all simple graphs $h$ with vertex set $I$ and with the same or
more edges than $g$.
Let $I=S_1\sqcup S_2$, $g\in\rG[I]$, and $g_i\in\rG[S_i]$, $i=1,2$. 
From~\eqref{e:prod-g} and~\eqref{e:coprod-g} one obtains
\begin{gather}\label{e:lambda-prod-graph}
\mu_{S_1,S_2}(p_{g_1}\otimes p_{g_2}) = p_{g_1\sqcup g_2}\\
\Delta_{S_1,S_2}(p_g) =
\begin{cases}
p_{g|_{S_1}} \otimes p_{g|_{S_2}}  & \text{if no edge of $g$ connects $S_1$ to $S_2$,} \\
0 & \text{otherwise.}
\end{cases}
\end{gather}
Formula~\eqref{e:lambda-prod-graph}
implies that $\bG$ is the free commutative monoid on the species of connected
graphs.
From~\eqref{e:s-model} we deduce that the morphism $\phi$ of Proposition~\ref{p:f-model} takes the following form 
on these bases:
\[
\phi(\ell\otimes p_g) = \sum_{U\in\Un(I,\ell):\, g(U)=g} \lambda_U.
\]

\section{Applications}\label{s:additional}

We conclude with some applications and remarks regarding past and future work.

\subsection{Counting conjugacy classes}\label{ss:counting}

Let $k_n(q)$ denote the number of conjugacy classes of the group of unitriangular matrices of size $n$ with entries in the field with $q$ elements.
Higman's conjecture states that, for fixed $n$, $k_n$
is a polynomial function of $q$.
Much effort has been devoted to the precise determination
of these numbers or their asymptotic behavior~\cite{Goo:2006,GR:2009,Hig:1960,Rob:1998,VA:2003,VAOV:2008}.

We fix $q$ and let $n$ vary. It turns out that the existence of a Hopf monoid structure
on class functions imposes certain linear conditions on the sequence $k_n(q)$,
as we explain next.

Given a finite-dimensional Hopf monoid $\bh$, consider the generating function
\begin{equation}\label{e:type}
\type{\bh}{x} = \sum_{n\geq 0} \dim_{\field} \bigl(\bh[n]_{\Sr_n}\bigr) x^n.
\end{equation}
Here $[n]$ denotes the set $\{1,2,\ldots,n\}$ and $\bh[n]_{\Sr_n}$
is the (quotient) space of coinvariants for the action of the symmetric group
(afforded by the species structure of $\bh$).

For example, since
\[
(\bL\times\bPi)[n]_{\Sr_n} = (\bL[n]\otimes\bPi[n])_{\Sr_n} \cong \bPi[n],
\]
we have
\begin{equation}\label{e:type-pi}
\type{\bL\times\bPi}{x} = \sum_{n\geq 0} B_n x^n,
\end{equation}
where $B_n$ is the \emph{$n$-th Bell number}, the number of partitions of the set $[n]$. 

On the other hand, from~\eqref{e:free-partition}, 
\[
\type{\bL\times\bPi}{x} = \type{\Tc(\bPi_a)}{x}.
\]
It is a general fact that, for a species $\bq$ with $\bq[\emptyset]=0$,
\[
\type{\Tc(\bq)}{x} = \frac{1}{1-\type{\bq}{x}}.
\]
(This follows for instance from~\cite[Theorem~2.b, Section 1.4]{BLL:1998}.)
Therefore,
\begin{equation}\label{e:type-pi2}
\type{\bL\times\bPi}{x} = \frac{1}{\displaystyle 1-\sum_{n\geq 1} A_n x^n}
\end{equation}
where $A_n$ is the number of atomic partitions of the linearly ordered set $[n]$.

From~\eqref{e:type-pi} and~\eqref{e:type-pi2} we deduce that
\[
\sum_{n\geq 0} B_n x^n = \frac{1}{\displaystyle 1-\sum_{n\geq 1} A_n x^n}
\]
a fact known from~\cite{BZ:2009}.

\smallskip

Consider now the injections
\[
\SC(\Un) \into \CF(\Un) \qand \bL\times\bPi \into \SC(\Un).
\]
Both are morphisms of Hopf monoids (Propositions~\ref{p:super-unit} and~\ref{p:s-model}).
\emph{Lagrange's theorem} for Hopf monoids implies in this situation that both quotients
\[
\frac{\type{\CF(\Un)}{x}}{\type{\SC(\Un)}{x}}
\qand
\frac{\type{\SC(\Un)}{x}}{\type{\bL\times\bPi}{x}}
\]
belong to $\Nb\llb x\rrb$, that is, have nonnegative (integer) coefficients~\cite[Corollary~13]{AL:2012}. In particular,
\[
\frac{\type{\CF(\Un)}{x}}{\type{\bL\times\bPi}{x}} \in \Nb\llb x\rrb
\]
as well. 

We have
\[
\CF(\Un)[n]_{\Sr_n} = \Biggl(\bigoplus_{\ell\in\rL[n]}\CF\bigl(\Un([n],\ell)\bigr)\Biggr)_{\Sr_n} \cong \CF\bigl(\Un([n])\bigr).
\]
Therefore,
\[
\type{\CF(\Un)}{x} = \sum_{n\geq 0} k_n(q) x^n.
\]

By combining the above, we deduce
\[
\Bigl(\sum_{n\geq 0} k_n(q) x^n\Bigr)\Bigl(1-\sum_{n\geq 1} A_n x^n\Bigr) \in \Nb\llb x\rrb,
\]
whence the following result.

\begin{corollary}\label{c:counting}
The following linear inequalities are satisfied for every $n\in\Nb$ and every prime power $q$.
\begin{equation}\label{e:counting}
k_n(q) \geq \sum_{i=0}^{n-1} A_{n-i}\, k_i(q).
\end{equation}
\end{corollary}

For instance, for $n=8$, the inequality is
\[
k_6(q) \geq 92 + 22 k_1(q) + 6 k_2(q) + 2 k_3(q) + 
k_4(q) + k_5(q).
\]

Inequality~\eqref{e:counting} is stronger than merely stating that there are more conjugacy classes than superclasses. 
For instance, for $q=2$ and $n=6$, the right hand side of the
inequality is $213$ (provided we
use the correct values for $k_i(2)$ for $i\leq 5$), while there are only $B_6=203$ superclasses.  The first few values of the sequence $k_n(2)$ appear
in~\cite{Slo:oeis} as A007976; in particular, $k_6(2)=275$.

The numbers $k_n(q)$ are known for $n\leq 13$ from work of 
Arregi and Vera-L\'opez~\cite{VA:1992,VA:1995,VA:2003}; see also~\cite{VAOV:2008}. (There is an incorrect sign in the value given for $k_7(q)$ in~\cite[page~923]{VA:1995}: the lowest term should be $-7q$.)

We may derive additional information on these numbers from the injective
morphism of Hopf monoids (Proposition~\ref{p:class-unit})
\[
\CF(\Un)\into \FC(\Un).
\]
Define a sequence of integers $c_n(q)$, $n\geq 1$, by means of
\begin{equation}\label{e:kc}
\sum_{n\geq 0} k_n(q) x^n = \frac{1}{\displaystyle 1-\sum_{n\geq 1} c_n(q) x^n}.
\end{equation}
Arguing as above we obtain the following result.
\begin{corollary}\label{c:counting2}
The following linear inequalities are satisfied for every $n\in\Nb$ and every prime power $q$.
\begin{equation}\label{e:counting2}
q^{\binom{n}{2}} \geq \sum_{i=1}^{n} q^{\binom{n-i}{2}}\, c_i(q).
\end{equation}
\end{corollary}

Through~\eqref{e:kc}, these inequalities impose further constraints on the
numbers $k_n(q)$.

The first few values of the sequence $c_n(q)$ are as follows, with $t=q-1$.
\begin{align*}
c_1(q) &= 1\\
c_2(q) & = t\\
c_3(q) & = t^2 +t \\
c_4(q) & = 2t^3 +4t^2 +t \\
c_5(q) & = 5t^4 +14t^3 +9t^2 +t\\
c_6(q) &= t^6 +18t^5 +55t^4 +54t^3 +16t^2 +t
\end{align*}

\begin{conjecture}\label{cj:stronghigman}
There exists polynomials $p_n(t)\in\Nb[t]$ such that $c_n(q) = p_n(q-1)$ for
every prime power $q$ and every $n\geq 1$.
\end{conjecture}

Using the formulas given by Vera-L\'opez et al in~\cite[Corollaries~10-11]{VAOV:2008} for computing $k_n(q)$,
we have verified the conjecture for $n\leq 13$.

Polynomiality of $k_n(q)$ is equivalent to that of $c_n(q)$. On the other hand,
the nonnegativity of $c_n$ as a polynomial of $t$ implies that of $k_n$,
but not conversely.
Thus Conjecture~\ref{cj:stronghigman} is a strong form of Higman's.

It is possible to show, using the methods of~\cite{AM:2012}, that the monoid $\CF(\Un)$ is free.
This implies that the integers
$c_n(q)$ are nonnegative for every $n\geq 1$ and prime power $q$.


\subsection{From Hopf monoids to Hopf algebras}\label{ss:hopfalg}

It is possible to associate a number of graded Hopf algebras to
a given Hopf monoid $\bh$. This is the subject of~\cite[Part~III]{AM:2010}.
In particular, there are two graded Hopf algebras $\Kc(\bh)$
and $\Kcb(\bh)$ related by a canonical surjective morphism
\[
\Kc(\bh) \onto \Kcb(\bh).
\]
The underlying spaces of these Hopf algebras are
\[
\Kc(\bh) = \bigoplus_{n\geq 0} \bh[n]
\qand
\Kcb(\bh) = \bigoplus_{n\geq 0} \bh[n]_{\Sr_n},
\]
where $\bh[n]_{\Sr_n}$ is as in~\eqref{e:type}.
The product and coproduct
of these Hopf algebras is built from those of the Hopf monoid $\bh$ together
with certain canonical transformations. The latter involve certain combinatorial
procedures known as \emph{shifting} and \emph{standardization}.
For more details, we refer to~\cite[Chapter~15]{AM:2010}.

For example, one has that
\[
\Kcb(\bL)=\field[x]
\]
is the polynomial algebra on one primitive generator, while $\Kc(\bL)$
is the Hopf algebra introduced by Patras and Reutenauer in~\cite{PR:2004}.

According to~\cite[Section~17.4]{AM:2010}, $\Kcb(\bPi)$ is the ubiquitous
Hopf algebra of \emph{symmetric functions}, while 
$\Kc(\bPi)$ is the Hopf algebra of \emph{symmetric functions in noncommuting variables}, an object studied in various
references including~\cite[Section~6.2]{AM:2006},~\cite{BHRZ:2006,BZ:2009,RS:2006}.

For any Hopf monoid $\bh$, one has~\cite[Theorem~15.13]{AM:2010}
\[
\Kcb(\bL\times\bh) \cong \Kc(\bh).
\]
Combining with Corollary~\ref{c:super-structure} we obtain that, when the
field of coefficients is $\Fb_2$,
\[
\Kcb\bigl(\SC(\Un)\bigr) \cong \Kcb\bigl(\bL\times\bPi\bigr) \cong \Kc(\bPi).
\]
In other words, the Hopf algebra constructed from superclass functions
on unitriangular matrices (with entries in $\Fb_2$) via the functor $\Kcb$ is isomorphic
to the Hopf algebra of symmetric functions in noncommuting variables.
This is the main result of~\cite{AIM:2012}.

The freeness of the Hopf algebra $\Kc(\bPi)$, a fact known from~\cite{Har:1978,Wol:1936},
is a consequence of Proposition~\ref{p:free-sc}.

We mention that one may arrive at Corollary~\ref{c:counting} by employing
the Hopf algebra $\Kcb\bigl(\CF(\Un)\bigr)$ (rather than the Hopf monoid
$\CF(\Un)$) and appealing to Lagrange's theorem for graded connected Hopf algebras.

\subsection{Supercharacters and beyond}\label{ss:super-char}

The notion of superclass on a unitriangular group comes with a companion
notion of \emph{supercharacter}, and a full-fledged theory relating them.
This is due to the pioneering work of Andr\'e~\cite{And:1995,And:1995-2} and 
later Yan~\cite{Yan:2001}. Much of this theory extends to 
algebra groups~\cite{And:1999,DI:2008,DT:2009}. More recently,
a connection with classical work on Schur rings has been understood~\cite{Hen:2010}.

In regards to the object of present interest, the Hopf monoid $\SC(\Un)$,
this implies the existence of a second canonical linear basis,
consisting of supercharacters. The work of Andr\'e and Yan provides a
\emph{character formula}, which yields the change of basis between
superclass functions and supercharacters. 
We plan to study the Hopf monoid structure of $\SC(\Un)$ on the supercharacter basis in future work.

\section*{Appendix. On free Hopf algebras and Hopf monoids}\label{s:appendix}

A free algebra may carry several Hopf algebra structures. It always carries
a canonical one in which the generators are primitive. It turns out that
under certain conditions, any Hopf structure on a free algebra is isomorphic to the canonical one. We provide such a result below.
An analogous result holds for Hopf monoids in vector species. This 
is applied in the paper in Section~\ref{s:free}.

We assume that the base field $\field$ is of characteristic $0$.

We employ the {\em first Eulerian idempotent}~\cite{GS:1991},
~\cite[Section 4.5.2]{Lod:1998},~\cite[Section 8.4]{Reu:93}. For any connected Hopf algebra $H$, the identity map $\id:H\to H$
is locally unipotent with respect to the convolution product of $\End(H)$. 
Therefore,
\begin{equation}\label{e:euler}
\euler:=\log(\id)=\sum_{k\geq 1}\frac{(-1)^{k+1}}{k}(\id-\iota\epsilon)^{\ast k}
\end{equation}
is a well-defined linear endomorphism of $H$. Here 
\[
\iota:\field\to H
\qand 
\epsilon:H\to\field
\]
denote the unit and counit maps of $H$ respectively, and the powers are
with respect to the convolution product.
It is an important fact that if $H$ is in addition
cocommutative, then $\euler(x)$ is a primitive element of $H$ for any $x\in H$.
In fact, the operator $\euler$ is in this case a projection onto the space of primitive elements~\cite{Pat:1994},~\cite[pages 314-318]{Sch:94}.

Let $T(V)$ denote the free algebra on a vector space $V$:
\[
T(V)=\bigoplus_{n\geq 0} V^{\otimes n}.
\]
The product is concatenation of tensors. 
We say in this case that $V$ freely generates.

The unique morphisms of algebras
\[
\Delta:T(V)\to T(V)\otimes T(V)
\qand
\epsilon:T(V)\to\field
\]
given for all $v\in V$ by
\[
\Delta(v) = 1\otimes v + v\otimes 1
\qand 
\epsilon(v)=0
\]
turn $T(V)$ into a connected, cocommutative Hopf algebra.
This is the \emph{canonical} Hopf structure on $T(V)$.

\begin{proposition}\label{p:freealg}
Let $\field$ be a field of characteristic $0$.
Let $H$ be a connected cocommutative Hopf algebra over $\field$.
Suppose $H\cong T(W)$ as algebras, in such a way that the image of $W$ lies in the kernel of the $\epsilon$.
Then there exists a (possibly different) isomorphism
of Hopf algebras $H\cong T(W)$, where the latter is
endowed with its canonical Hopf structure.
\end{proposition}
\begin{proof}
We may assume $H=T(W)$ as algebras, for some subspace $W$ of 
$\ker(\epsilon)$. Since $H$ is connected and $\field$ is of characteristic $0$, the Eulerian idempotent $\euler$ is defined. Let $V=\euler(W)$. We show below that $V\cong W$ and that $V$ freely generates $H$. Since $H$ is cocommutative, $V$ consists of primitive elements,
and therefore $H\cong T(V)$ as Hopf algebras. This completes the proof.

Let 
\[
H_+ = \bigoplus_{n\geq 1} W^{\otimes n}.
\]
Since $\epsilon$ is a morphism of algebras, $H_+\subseteq \ker(\epsilon)$,
and since both spaces are of codimension $1$, they must agree: 
$H_+= \ker(\epsilon)$.

Define $\Delta_+:H_+ \to H_+\otimes H_+$ by
\[
\Delta_+(x)= \Delta(x) -1\otimes x - x\otimes 1.
\]
By counitality, 
\[
(\epsilon\otimes\id)\Delta_+ = 0 = (\id\otimes\epsilon)\Delta_+.
\]
Therefore, $\Delta_+(H_+)\subseteq \ker(\epsilon)\otimes\ker(\epsilon) = H_+\otimes H_+$, and hence
\[
\Delta_+^{(k-1)}(H_+) \subseteq H_+^{\otimes k}
\]
for all $k\geq 1$. In addition, since $H=T(W)$ as algebras,
\[
\mu^{(k-1)}(H_+^{\otimes k}) \subseteq \sum_{n\geq 2} W^{\otimes n}
\]
for all $k\geq 2$.

Take $w\in W$. Then
\begin{align*}
\euler(w) &= \sum_{k\geq 1}\frac{(-1)^{k+1}}{k}(\id-\iota\epsilon)^{\ast k}(w)= w + \sum_{k\geq 2}\frac{(-1)^{k+1}}{k} \mu^{(k-1)}\Delta_+^{(k-1)}(w)\\
&\equiv w + \sum_{n\geq 2} W^{\otimes n}.
\end{align*}
By triangularity, $\euler:W\to V$ is invertible and hence $V$ generates $H$.

Now take $w_1$ and $w_2\in W$. It follows from the above that
\[
\euler(w_1)\euler(w_2) \equiv w_1w_2 + \sum_{n\geq 3} W^{\otimes n},
\]
and a similar triangular relation holds for higher products. Hence $V$ generates $H$ freely.
\end{proof}

\smallskip

The Eulerian idempotent is defined for connected Hopf monoids in species,
by the same formula as~\eqref{e:euler}. 
Let $\bp$ be a species such that $\bp[\emptyset]=0$.
The free monoid $\Tc(\bp)$ and its canonical Hopf structure is discussed
in~\cite[Section~11.2]{AM:2010}.
The arguments in Proposition~\ref{p:freealg}
may easily be adapted to this setting to yield the following result.

\begin{proposition}\label{p:freemon}
Let $\field$ be a field of characteristic $0$.
Let $\bh$ be a connected cocommutative Hopf monoid in vector species over $\field$.
Suppose $\bh\cong\Tc(\bp)$ as monoids for some species $\bp$ such that 
$\bp[\emptyset]=0$.
Then there exists a (possibly different) isomorphism of Hopf monoids
$\bh\cong\Tc(\bp)$, where the latter is endowed with its canonical Hopf structure.
\end{proposition}


\bibliographystyle{abbrv}  
\bibliography{unitriang}

\begin{thebibliography}{10}

\bibitem{AIM:2012}
M.~Aguiar, C.~Andr\'e, C.~Benedetti, N.~Bergeron, Z.~Chen, P.~Diaconis,
  A.~Hendrickson, S.~Hsiao, I.~M. Isaacs, A.~Jedwab, K.~Johnson, G.~Karaali,
  A.~Lauve, T.~Le, S.~Lewis, H.~Li, K.~Magaard, E.~Marberg, J.-C. Novelli,
  A.~Pang, F.~Saliola, L.~Tevlin, J.-Y. Thibon, N.~Thiem, V.~Venkateswaran,
  C.~R. Vinroot, N.~Yan, and M.~Zabrocki.
\newblock Supercharacters, symmetric functions in noncommuting variables, and
  related {H}opf algebras.
\newblock {\em Advances in Mathematics}, 229(4):2310 -- 2337, 2012.

\bibitem{AL:2012}
M.~Aguiar and A.~Lauve.
\newblock Lagrange's theorem for {H}opf monoids in species, 2012.
\newblock To appear in Canadian Journal of Mathematics.

\bibitem{AM:2006}
M.~Aguiar and S.~Mahajan.
\newblock {\em Coxeter groups and {H}opf algebras}, volume~23 of {\em Fields
  Inst. Monogr.}
\newblock Amer. Math. Soc., Providence, RI, 2006.

\bibitem{AM:2010}
M.~Aguiar and S.~Mahajan.
\newblock {\em Monoidal functors, species and {H}opf algebras}, volume~29 of
  {\em CRM Monograph Series}.
\newblock American Mathematical Society, Providence, RI, 2010.

\bibitem{AM:2012}
M.~Aguiar and S.~Mahajan.
\newblock On the {H}adamard product of {H}opf monoids, 2012.
\newblock Available at \url{arXiv:1209.1363}.

\bibitem{And:1995}
C.~A.~M. Andr{\'e}.
\newblock Basic characters of the unitriangular group.
\newblock {\em J. Algebra}, 175(1):287--319, 1995.

\bibitem{And:1995-2}
C.~A.~M. Andr{\'e}.
\newblock Basic sums of coadjoint orbits of the unitriangular group.
\newblock {\em J. Algebra}, 176(3):959--1000, 1995.

\bibitem{And:1999}
C.~A.~M. Andr{\'e}.
\newblock Irreducible characters of finite algebra groups.
\newblock In {\em Matrices and group representations ({C}oimbra, 1998)},
  volume~19 of {\em Textos Mat. S\'er. B}, pages 65--80. Univ. Coimbra,
  Coimbra, 1999.

\bibitem{BLL:1998}
F.~Bergeron, G.~Labelle, and P.~Leroux.
\newblock {\em Combinatorial species and tree-like structures}, volume~67 of
  {\em Encyclopedia of Mathematics and its Applications}.
\newblock Cambridge University Press, Cambridge, 1998.
\newblock Translated from the 1994 French original by Margaret Readdy, with a
  foreword by Gian-Carlo Rota.

\bibitem{BHRZ:2006}
N.~Bergeron, C.~Hohlweg, M.~Rosas, and M.~Zabrocki.
\newblock Grothendieck bialgebras\textup, partition lattices\textup, and
  symmetric functions in noncommutative variables.
\newblock {\em Electron. J. Combin.}, 13(1):R75 (electronic), 2006.

\bibitem{BZ:2009}
N.~Bergeron and M.~Zabrocki.
\newblock The {H}opf algebras of symmetric functions and quasi-symmetric
  functions in non-commutative variables are free and co-free.
\newblock {\em J. Algebra Appl.}, 8(4):581--600, 2009.

\bibitem{DI:2008}
P.~Diaconis and I.~M. Isaacs.
\newblock Supercharacters and superclasses for algebra groups.
\newblock {\em Trans. Amer. Math. Soc.}, 360(5):2359--2392, 2008.

\bibitem{DT:2009}
P.~Diaconis and N.~Thiem.
\newblock Supercharacter formulas for pattern groups.
\newblock {\em Trans. Amer. Math. Soc.}, 361(7):3501--3533, 2009.

\bibitem{GS:1991}
M.~Gerstenhaber and S.~D. Schack.
\newblock The shuffle bialgebra and the cohomology of commutative algebras.
\newblock {\em J. Pure Appl. Algebra}, 70(3):263--272, 1991.

\bibitem{Goo:2006}
S.~M. Goodwin.
\newblock On the conjugacy classes in maximal unipotent subgroups of simple
  algebraic groups.
\newblock {\em Transform. Groups}, 11(1):51--76, 2006.

\bibitem{GR:2009}
S.~M. Goodwin and G.~R{\"o}hrle.
\newblock Calculating conjugacy classes in {S}ylow {$p$}-subgroups of finite
  {C}hevalley groups.
\newblock {\em J. Algebra}, 321(11):3321--3334, 2009.

\bibitem{Har:1978}
V.~K. Har{\v{c}}enko.
\newblock Algebras of invariants of free algebras.
\newblock {\em Algebra i Logika}, 17(4):478--487, 491, 1978.

\bibitem{Hen:2010}
A.~O.~F. Hendrickson.
\newblock Supercharacter theories and {S}chur rings, available 2010 at
  \url{arXiv:1006.1363v1}.

\bibitem{Hig:1960}
G.~Higman.
\newblock Enumerating {$p$}-groups. {I}. {I}nequalities.
\newblock {\em Proc. London Math. Soc. (3)}, 10:24--30, 1960.

\bibitem{Isa:1995}
I.~M. Isaacs.
\newblock Characters of groups associated with finite algebras.
\newblock {\em J. Algebra}, 177(3):708--730, 1995.

\bibitem{Kir:1995}
A.~A. Kirillov.
\newblock Variations on the triangular theme.
\newblock In {\em Lie groups and {L}ie algebras: {E}. {B}. {D}ynkin's
  {S}eminar}, volume 169 of {\em Amer. Math. Soc. Transl. Ser. 2}, pages
  43--73. Amer. Math. Soc., Providence, RI, 1995.

\bibitem{Lod:1998}
J.-L. Loday.
\newblock {\em Cyclic homology}, volume 301 of {\em Grundlehren Math. Wiss.}
\newblock Springer, Berlin, 2nd edition, 1998.

\bibitem{Pat:1994}
F.~Patras.
\newblock L'alg\`ebre des descentes d'une big\`ebre gradu\'ee.
\newblock {\em J. Algebra}, 170(2):547--566, 1994.

\bibitem{PR:2004}
F.~Patras and C.~Reutenauer.
\newblock On descent algebras and twisted bialgebras.
\newblock {\em Mosc. Math. J.}, 4(1):199--216, 2004.

\bibitem{Reu:93}
C.~Reutenauer.
\newblock {\em Free {L}ie algebras}, volume~7 of {\em London Math. Soc. Monogr.
  (N.S.)}.
\newblock The Clarendon Press, Oxford Univ. Press, New York, 1993.

\bibitem{Rob:1998}
G.~R. Robinson.
\newblock Counting conjugacy classes of unitriangular groups associated to
  finite-dimensional algebras.
\newblock {\em J. Group Theory}, 1(3):271--274, 1998.

\bibitem{RS:2006}
M.~H. Rosas and B.~E. Sagan.
\newblock Symmetric functions in noncommuting variables.
\newblock {\em Trans. Amer. Math. Soc.}, 358(1):215--232 (electronic), 2006.

\bibitem{Sch:94}
W.~R. Schmitt.
\newblock Incidence {H}opf algebras.
\newblock {\em J. Pure Appl. Algebra}, 96(3):299--330, 1994.

\bibitem{Slo:oeis}
N.~J.~A. Sloane.
\newblock The on-line encyclopedia of integer sequences.
\newblock Published electronically at {\tt
  www.research.att.com/$\scriptstyle\sim$njas/sequences/}, OEIS.

\bibitem{VA:1992}
A.~Vera-L{\'o}pez and J.~M. Arregi.
\newblock Conjugacy classes in {S}ylow {$p$}-subgroups of {${\rm GL}(n,q)$}.
\newblock {\em J. Algebra}, 152(1):1--19, 1992.

\bibitem{VA:1995}
A.~Vera-L{\'o}pez and J.~M. Arregi.
\newblock Some algorithms for the calculation of conjugacy classes in the
  {S}ylow {$p$}-subgroups of {${\rm GL}(n,q)$}.
\newblock {\em J. Algebra}, 177(3):899--925, 1995.

\bibitem{VA:2003}
A.~Vera-L{\'o}pez and J.~M. Arregi.
\newblock Conjugacy classes in unitriangular matrices.
\newblock {\em Linear Algebra Appl.}, 370:85--124, 2003.

\bibitem{VAOV:2008}
A.~Vera-L{\'o}pez, J.~M. Arregi, L.~Ormaetxea, and F.~J. Vera-L{\'o}pez.
\newblock The exact number of conjugacy classes of the {S}ylow {$p$}-subgroups
  of {${\rm GL}(n,q)$} modulo {$(q-1)^{13}$}.
\newblock {\em Linear Algebra Appl.}, 429(2-3):617--624, 2008.

\bibitem{Wol:1936}
M.~C. Wolf.
\newblock Symmetric functions of non-commutative elements.
\newblock {\em Duke Math. J.}, 2(4):626--637, 1936.

\bibitem{Yan:2001}
N.~Yan.
\newblock {\em Representation theory of the finite unipotent linear groups}.
\newblock PhD thesis, University of Pennsylvania, 2001.

\end{thebibliography}

\end{document}